\documentclass[11pt,a4paper]{amsart}
\usepackage{fullpage}
\usepackage{amsmath}
\usepackage{amsthm}
\usepackage{amssymb}
\usepackage{mathtools}
\usepackage[only,mapsfrom]{stmaryrd}
\usepackage{graphicx}
\usepackage{cancel} 
\usepackage{color} 
\usepackage[all]{xy} 
\usepackage{tikz}
\usetikzlibrary{arrows,calc,cd}
\usepackage{nicefrac}
\usepackage{paralist}
\usepackage{enumitem}
\usepackage[normalem]{ulem}
\usepackage[colorlinks,citecolor=blue,linkcolor=blue,urlcolor=blue,filecolor=blue,breaklinks]{hyperref}

\DeclareRobustCommand{\SkipTocEntry}[5]{}
\numberwithin{equation}{section}

\newenvironment{itemizeb}%
{\begin{compactitem}

}%
{\end{compactitem}}

\newtheorem{thm}{Theorem}[section]
\newtheorem{athm}{Theorem}
\newtheorem{acor}[athm]{Corollary}
\newtheorem{lem}[thm]{Lemma}
\newtheorem{prop}[thm]{Proposition}
\newtheorem{cor}[thm]{Corollary}

\newtheorem{obs}[thm]{Observation}
\newtheorem*{thm*}{Theorem}
\newtheorem*{conj*}{Conjecture}
\newtheorem*{cor*}{Corollary}
\newtheorem*{ques*}{Question}

\theoremstyle{definition}
\newtheorem{ex}[thm]{Example}
\newtheorem{rem}[thm]{Remark}
\newtheorem{defn}[thm]{Definition}
\newtheorem{notation}[thm]{Notation}
\newtheorem{construction}[thm]{Construction}
\newtheorem*{rem*}{Remark}

\newcommand{\cA}{\mathcal{A}}

\newcommand{\cC}{\mathcal{C}}

\newcommand{\cE}{\mathcal{E}}
\newcommand{\cF}{\mathcal{F}}
\newcommand{\cG}{\mathcal{G}}

\newcommand{\cM}{\mathcal{M}}

\newcommand{\cU}{\mathcal{U}}

\newcommand{\bD}{\mathbb{D}}

\newcommand{\bI}{\mathbb{I}}

\newcommand{\bN}{\mathbb{N}}

\newcommand{\bR}{\mathbb{R}}
\newcommand{\bS}{\mathbb{S}}

\newcommand{\bZ}{\mathbb{Z}}

\DeclareMathOperator{\Lk}{Lk}
\DeclareMathOperator{\Fix}{Fix}
\DeclareMathOperator{\Map}{Map}
\DeclareMathOperator{\PMap}{PMap}
\DeclareMathOperator{\Diff}{Diff}
\DeclareMathOperator{\Homeo}{Homeo}
\DeclareMathOperator{\Ends}{Ends}
\newcommand{\Hom}{{\mathrm{Hom}}}
\newcommand{\Aut}{{\mathrm{Aut}}}
\newcommand{\id}{{{\rm id}}}
\newcommand{\Emb}{\ensuremath{\mathrm{Emb}}}
\newcommand{\Emba}{\ensuremath{\mathrm{Emb}^{\mathrm{adm}}}}

\newcommand{\incl}[3][right]%
{%
\draw[<-,>=#1 hook] #2 to ($ #2!0.5!#3 $);
\draw[->,>=stealth'] ($ #2!0.5!#3 $) to #3;%
}
\newcommand{\inclusion}[5][right]%
{%
\draw[<-,>=#1 hook] #4 to ($ #4!0.5!#5 $) node[#2,font=\small]{#3};
\draw[->,>=stealth'] ($ #4!0.5!#5 $) to #5;%
}

\newenvironment{inlinecond}[1]{%
  \begin{enumerate}%
  \renewcommand*{\theenumi}{#1}%
  \renewcommand*{\labelenumi}{#1}%
  \item%
}{%
  \end{enumerate}%
}

\title{On the homology of big mapping class groups}

\author{Martin Palmer}
\address{Institutul de Matematică Simion Stoilow al Academiei Române, 21 Calea Griviței, 010702 Bucharest, Romania}
\email{mpanghel@imar.ro}

\author{Xiaolei Wu}
\address{Shanghai Center for Mathematical Sciences, Jiangwan Campus, Fudan University, No.2005 Songhu Road, Shanghai, 200438, P.R. China}
\email{xiaoleiwu@fudan.edu.cn}

\date{5 September 2024}

\begin{document}

\begin{abstract}
We prove that the mapping class group of the one-holed Cantor tree surface is acyclic. This in turn determines the homology of the mapping class group of the once-punctured Cantor tree surface (i.e.~the plane minus a Cantor set), in particular answering a recent question of Calegari and Chen. We in fact prove these results for a general class of infinite-type surfaces called binary tree surfaces. To prove our results we use two main ingredients: one is a modification of an argument of Mather related to the notion of \emph{dissipated groups}; the other is a general homological stability result for mapping class groups of infinite-type surfaces.
\end{abstract}
\maketitle

\section*{Introduction}

There is a long tradition of calculating the homology of mapping class groups of surfaces. Most notably, Arnold \cite{Ar69} calculated the homology of the pure braid groups; Harer \cite{Har85} proved that the mapping class groups of compact, connected, orientable surfaces satisfy homological stability under various surface operations; Madsen and Weiss \cite{MW07} calculated the stable homology of the mapping class groups of compact, connected, orientable surfaces, in particular confirming the Mumford conjecture \cite{Mumford1983}.

More recently, there has been a wave of interest in studying \emph{big mapping class groups} (mapping class groups of infinite-type surfaces); see \cite{AV20} for a recent survey. However, there are so far relatively few studies of the homology of big mapping class groups, and almost none in degrees above one. The first cohomology of big pure mapping class groups $\PMap(S)$ has been calculated (and is countable) when the genus of $S$ is positive and has been shown to be uncountable when the genus of $S$ is zero; see \cite{APV20,DomatPlummer2020}. On the other hand, the first \emph{homology} group of $\PMap(S)$ is uncountable whenever $S$ has at most one non-planar end; see \cite{Domat2022}. The first homology of the full mapping class group $\Map(S)$ has been calculated in certain cases: if $\cC$ is a Cantor set in a \emph{finite type} surface $S$, then $H_1(\Map(S \smallsetminus \cC)) \cong H_1(\Map(S))$ \cite{CalegariChen2022} (see also \cite{Vlamis21} for three special cases). On the other hand, the first homology of $\Map(\bR^2 \smallsetminus \bN)$ has been shown to be uncountable; see \cite{MalesteinTao21}. Finally, one earlier calculation of \emph{higher} homology groups of big mapping class groups of which we are aware is $H_2(\Map(\bS^2 \smallsetminus \cC)) \cong \bZ/2$ \cite{CC21}.

The purpose of the present paper is to initiate a systematic study of the homology of big mapping class groups $\Map(S)$.\footnote{Recall that, by definition, $\Map(S) = \pi_0(\Homeo_\partial(S))$, where $\Homeo_\partial(S)$ denotes the topological group of self-homeomorphisms of $S$ restricting to the identity on $\partial S$, equipped with the compact-open topology.} Surprisingly, we are able to completely determine the homology of $\Map(S)$ in all degrees for many infinite-type surfaces $S$, including for example the once-punctured Cantor tree surface $S = \bR^2 \smallsetminus \cC$. In contrast, we prove in parallel work \cite{PalmerWu2} that the integral homology of many big mapping class groups is uncountable in all (positive) degrees.

\addtocontents{toc}{\SkipTocEntry}
\subsection*{Results.}

Our first result completely calculates the homology of the mapping class group of the plane minus a Cantor set $\bR^2 \smallsetminus \cC$. For degrees $1$ and $2$, it was previously known by \cite{Vlamis21,CalegariChen2022} that $H_1(\Map(\bR^2 \smallsetminus \cC)) = 0$ and that $H_2(\Map(\bR^2 \smallsetminus \cC))$ contains an element of infinite order.

\begin{athm}
\label{thm:plane-minus-Cantor-special-case}
For each integer $i\geq 0$, we have
\[
H_i(\Map(\bR^2 \smallsetminus \cC)) \cong \begin{cases}
    \bZ & \text{for } i \text{ even}, \\
    0 & \text{for } i \text{ odd}.
\end{cases}
\]
\end{athm}

\begin{rem}
Thus $H^2(\Map(\bR^2 \smallsetminus \cC)) \cong \bZ$, answering positively Question A.15 of \cite{CC21}.
\end{rem}

Theorem \ref{thm:plane-minus-Cantor-special-case} is a consequence of the following vanishing result, together with the fact that the mapping class group $\Map(\bD^2 \smallsetminus \cC)$ of the closed disc minus a Cantor set is a central extension of $\Map(\bR^2 \smallsetminus \cC)$ by the infinite cyclic group generated by a Dehn twist around the boundary.

\begin{athm}
\label{thm:disc-minus-Cantor-special-case}
The mapping class group $\Map(\bD^2 \smallsetminus \cC)$ is acyclic, i.e., $\widetilde{H}_*(\Map(\bD^2 \smallsetminus \cC)) = 0$.
\end{athm}

The \emph{Cantor tree surface} is the $2$-sphere minus a Cantor set, $\bS^2 \smallsetminus \cC$, and gets its name from its alternative realisation as the boundary of a closed regular neighbourhood of an infinite binary tree properly embedded in $\bR^3$. Theorem \ref{thm:disc-minus-Cantor-special-case} may therefore be stated as saying that the mapping class group of the \emph{one-holed Cantor tree surface} is acyclic and Theorem \ref{thm:plane-minus-Cantor-special-case} calculates the homology of the mapping class group of the \emph{once-punctured Cantor tree surface}. The one-holed Cantor tree surface $\bD^2 \smallsetminus \cC$ is also denoted by $D_\cC$ and is illustrated in Figure \ref{fig:Cantor-tree-surfaces}.

\addtocontents{toc}{\SkipTocEntry}
\subsection*{Tree surfaces.}

Theorems \ref{thm:plane-minus-Cantor-special-case} and \ref{thm:disc-minus-Cantor-special-case} both hold more generally for a class of surfaces that we call \emph{binary tree surfaces}. To define this, we first introduce a notion of \emph{graph surfaces}:

\begin{defn}[\emph{Graph surfaces}]
\label{def:graph-surfaces}
Let $\Gamma$ be a graph of finite valence with no self-loops and let $\Sigma$ be a connected surface without boundary. For $i\geq 0$, denote by $\Sigma_i$ the result of removing the interiors of $i$ pairwise disjoint closed discs from $\Sigma$. For each vertex $v$ of $\Gamma$, let $\Sigma_v = \Sigma_{\lvert v \rvert}$ equipped with a labelling of its boundary circles by the edges of $\Gamma$ incident to $v$. Consider the disjoint union $\bigsqcup_v \Sigma_v$. If $e$ is an edge of $\Gamma$ incident to $v$ and $w$, glue the boundary circle of $\Sigma_v$ labelled by $e$ by an orientation-preserving diffeomorphism to the boundary circle of $\Sigma_w$ labelled by $e$. The corresponding surface
\[
\mathrm{Gr}_\Gamma(\Sigma) := \bigsqcup_v \Sigma_v / {\sim}
\]
is the \emph{graph surface} associated to $\Gamma$ and $\Sigma$.
\end{defn}

As a special case of this, we have the following.

\begin{defn}[\emph{Linear and binary tree surfaces}]
\label{def:linear-binary-tree-surfaces}
Denote by $\mathfrak{B}$ the infinite rooted binary tree and by $\mathfrak{L}$ the infinite rooted unary tree (i.e., the tree with vertices $\bN$ and edges between pairs of consecutive natural numbers). Then we may take $\Gamma = \mathfrak{B}$ or $\Gamma = \mathfrak{L}$ in the previous definition. We call $\mathrm{Gr}_\mathfrak{B}(\Sigma)$ and $\mathrm{Gr}_\mathfrak{L}(\Sigma)$ the \emph{binary tree surface} and the \emph{linear surface} associated to $\Sigma$.
\end{defn}

See Figures \ref{fig:Cantor-tree-surfaces} and \ref{fig:Loch-Ness-monster-surface} for examples. The effects of the operations $\mathrm{Gr}_\mathfrak{B}(-)$ and $\mathrm{Gr}_\mathfrak{L}(-)$ on the spaces of ends of surfaces are discussed in \S\ref{ss:end-spaces}.

\begin{figure}[tb]
    \centering
    \includegraphics[scale=0.8]{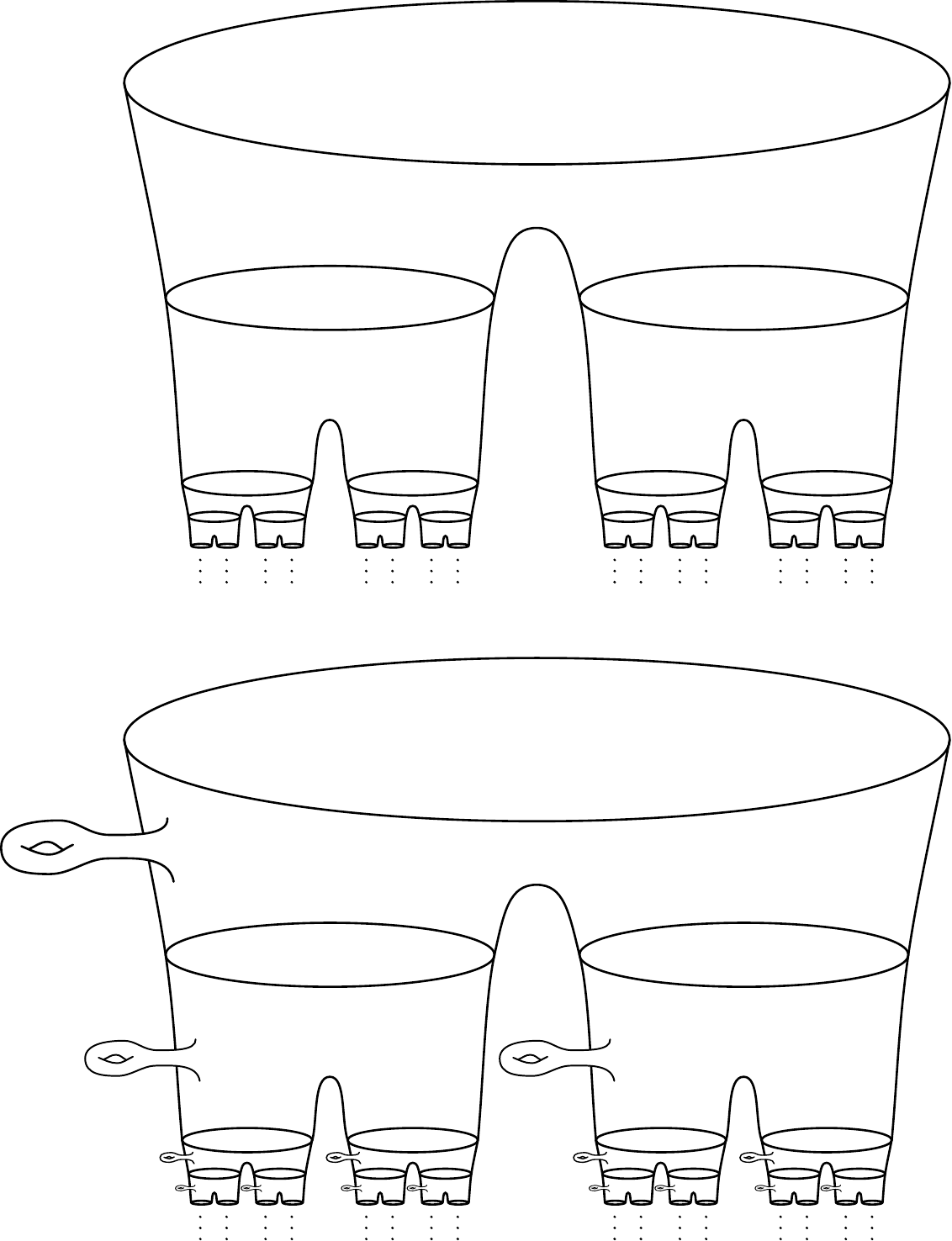}
    \caption{The surfaces $D_\cC$ and $B_\cC$, whose mapping class groups are acyclic by Theorem \ref{thm:disc-minus-Cantor}. Capping off the boundary (the top circle) with a disc results in the \emph{Cantor tree surface} $\mathrm{Gr}_\mathfrak{B}(\bS^2)$ and the \emph{blooming Cantor tree surface} $\mathrm{Gr}_\mathfrak{B}(T^2)$.}
    \label{fig:Cantor-tree-surfaces}
\end{figure}

\begin{notation}
\label{notation:one-holed-tree-surfaces}
If $\Sigma$ is a surface without boundary, we denote by $\Sigma_\circ$ the surface with circle boundary obtained by deleting the interior of a closed subdisc of $\Sigma$. We refer to this surface as the ``one-holed $\Sigma$''. Since we will mostly work with one-holed binary tree surfaces and one-holed linear surfaces, we introduce a shorter notation for these:
\[
\mathfrak{B}(\Sigma) := \mathrm{Gr}_\mathfrak{B}(\Sigma)_\circ \qquad\text{and} \qquad \mathfrak{L}(\Sigma) := \mathrm{Gr}_\mathfrak{L}(\Sigma)_\circ .
\]
\end{notation}

Key examples of graph surfaces, binary tree surfaces and linear surfaces are the following.

\begin{ex}
If $\Sigma = \bS^2$, then $\mathrm{Gr}_\Gamma(\bS^2)$ is homeomorphic to the boundary of a regular neighbourhood of an embedding of $\Gamma$ into $\bR^3$. In particular, if $\Gamma$ is finite, then $\mathrm{Gr}_\Gamma(\bS^2)$ is a closed, connected, orientable surface of genus equal to the first Betti number of $\Gamma$.

On the other hand, we have $\mathrm{Gr}_\mathfrak{L}(\bS^2) \cong \bR^2$ and the binary tree surface $\mathrm{Gr}_\mathfrak{B}(\bS^2)$ is the \emph{Cantor tree surface}, which is homeomorphic to $\bS^2 \smallsetminus \cC$, where $\cC$ is the Cantor set, as mentioned above. The binary tree surface $\mathrm{Gr}_\mathfrak{B}(T^2)$ is often called the \emph{blooming Cantor tree surface} and the linear surface $\mathrm{Gr}_\mathfrak{L}(T^2)$ is often called the \emph{Loch Ness monster surface}; see Figure \ref{fig:Loch-Ness-monster-surface}.
\end{ex}

\begin{notation}
As a shorthand, we denote the one-holed Cantor tree surface by $D_\cC := \mathfrak{B}(\bS^2) = \mathrm{Gr}_\mathfrak{B}(\bS^2)_\circ$ and similarly we denote the one-holed blooming Cantor tree surface by $B_\cC := \mathfrak{B}(T^2) = \mathrm{Gr}_\mathfrak{B}(T^2)_\circ$. These are illustrated in Figure \ref{fig:Cantor-tree-surfaces}.
\end{notation}

\addtocontents{toc}{\SkipTocEntry}
\subsection*{Results for tree surfaces.}

Our general vanishing theorem is the following.

\begin{athm}
\label{thm:disc-minus-Cantor}
Let $\Sigma$ be a connected surface without boundary. Then the mapping class group $\Map(\mathfrak{B}(\Sigma))$ is acyclic, i.e., $\widetilde{H}_*(\Map(\mathfrak{B}(\Sigma))) = 0$.
\end{athm}

In particular, taking $\Sigma = \bS^2$, we recover Theorem \ref{thm:disc-minus-Cantor-special-case}. As an immediate corollary, we obtain the following generalisation of Theorem \ref{thm:plane-minus-Cantor-special-case}.

\begin{acor}
\label{thm:plane-minus-Cantor}
Let $\Sigma$ be a connected surface without boundary. Then we have:
\[
H_i(\Map(\mathrm{Gr}_\mathfrak{B}(\Sigma) \smallsetminus pt)) \cong \begin{cases}
\bZ & \text{for } i \text{ even} \\
0 & \text{for } i \text{ odd.}
\end{cases}
\]
\end{acor}

\begin{proof}\belowdisplayskip=-2pt
This follows from Theorem \ref{thm:disc-minus-Cantor} and an argument using the Lyndon-Hochschild-Serre spectral sequence associated to the central extension
\[
1 \to \bZ \longrightarrow \Map(\mathfrak{B}(\Sigma)) = \Map(\mathrm{Gr}_\mathfrak{B}(\Sigma) \smallsetminus \mathring{\bD}^2) \longrightarrow \Map(\mathrm{Gr}_\mathfrak{B}(\Sigma) \smallsetminus pt) \to 1.
\]\qedhere
\end{proof}

\begin{rem}
\label{rmk:uncountably-many-examples}
There are uncountably many different surfaces to which Theorem \ref{thm:disc-minus-Cantor} applies, since there are uncountably many different surfaces of the form $\mathfrak{B}(\Sigma)$. For example, when $\Sigma$ is the (unique) connected planar surface with end space $E$, then $\mathfrak{B}(\Sigma)$ is homeomorphic to $\bD^2 \smallsetminus \Upsilon^\cC(E)$, whose end space $\Upsilon^\cC(E)$ is the \emph{Cantor compactification} of $E\omega$ (see \S\ref{ss:end-spaces}). The spaces $\Upsilon^\cC([0,\omega^\alpha])$ are pairwise non-homeomorphic for different countable ordinals $\alpha$ (as explained in Remark \ref{rmk:Cantor-compactifications}), so $\mathfrak{B}(\bS^2 \smallsetminus [0,\omega^\alpha])$, varying $\alpha$, is an uncountable family of pairwise non-homeomorphic surfaces to which Theorem \ref{thm:disc-minus-Cantor} applies.
\end{rem}

\begin{rem}[\emph{Uncountable homology}]
\label{rem:uncountable-results}
To contrast Theorem \ref{thm:disc-minus-Cantor}, we prove in \cite{PalmerWu2} that the integral homology of many big mapping class groups is \emph{uncountable} in all positive degrees.
\end{rem}

\begin{figure}[tb]
    \centering
    \includegraphics{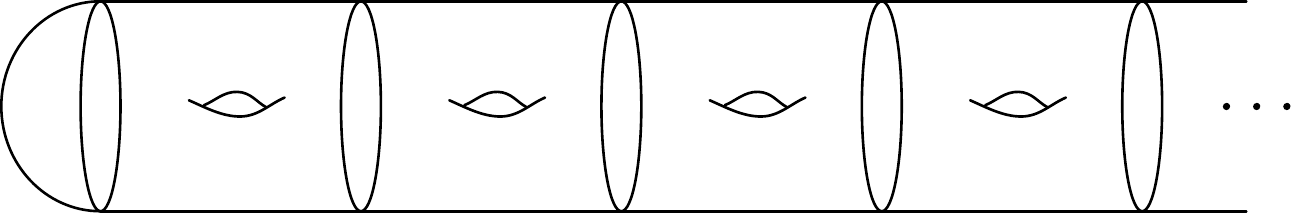}
    \caption{The Loch Ness monster surface $L = \mathrm{Gr}_\mathfrak{L}(T^2)$.}
    \label{fig:Loch-Ness-monster-surface}
\end{figure}

\addtocontents{toc}{\SkipTocEntry}
\subsection*{Homological stability.}

A key ingredient in the proof of Theorem \ref{thm:disc-minus-Cantor} is a homological stability result for big mapping class groups (Theorem \ref{thm:hom-stab}). To state this result fully, we need the notion of a \emph{topologically distinguished end} of a surface.

\begin{defn}
\label{defn:topologically-distinguished}
Let $E$ be a topological space. Two points $x,y \in E$ are called \emph{similar} if there are open neighbourhoods $U$ and $V$ of $x$ and $y$ respectively and a homeomorphism $U \cong V$ taking $x$ to $y$. This is an equivalence relation on $E$. A point $x \in E$ is called \emph{topologically distinguished} if its equivalence class under this relation is $\{x\}$, in other words it is similar only to itself.
\end{defn}

To apply this notion to the space of ends of a linear surface, we first observe:

\begin{rem}
If $E$ denotes the space of ends of $\Sigma$, then the space of ends of the linear surface $\mathfrak{L}(\Sigma)$ is homeomorphic to $(E\omega)^+$, the one-point compactification of the disjoint union $E\omega$ of countably infinitely many copies of $E$.
\end{rem}

\begin{defn}
\label{defn:telescope}
If $E$ denotes the space of ends of $\Sigma$ and the point at infinity of $(E\omega)^+$ is topologically distinguished, then the linear surface $\mathrm{Gr}_\mathfrak{L}(\Sigma)$ is called a \emph{telescope}. (This is not to be confused with the notion of a \emph{telescoping} surface from \cite[\S 3.2]{MannRafi}, which is quite different.)
\end{defn}

We note that not all linear surfaces are telescopes:

\begin{ex}
The \emph{flute surface} $\mathrm{Gr}_\mathfrak{L}(\bR^2) \cong \bR^2 \smallsetminus \bN$ is a telescope. In contrast, the Cantor tree surface $\mathrm{Gr}_\mathfrak{B}(\bS^2) \cong \mathrm{Gr}_\mathfrak{L}(\mathrm{Gr}_\mathfrak{B}(\bS^2))$ is a linear surface that is not a telescope, since $(\cC\omega)^+ \cong \cC$, which is homogeneous, so the point at infinity of $(\cC\omega)^+$ is not topologically distinguished.
\end{ex}

Now let $A$ and $X$ be connected surfaces with circle boundary and consider the sequence
\begin{equation}
\label{eq:stab-maps}
\cdots \to \Map(A \natural X^{\natural n}) \longrightarrow \Map(A \natural X^{\natural n+1}) \to \cdots
\end{equation}
of mapping class groups and stabilisation maps given by taking boundary connected sum with $X$ and extending homeomorphisms by the identity. We say that \eqref{eq:stab-maps} is \emph{homologically stable} if the $n$th map in the sequence induces isomorphisms on homology up to degree $f(n)$ and surjections on homology up to degree $f(n)+1$, for some diverging function $f \colon \bN \to \bN$.

\begin{athm}
\label{thm:hom-stab}
Let $A$ be any decorated surface. The sequence \eqref{eq:stab-maps} is homologically stable for
\begin{enumerate}[label=\textup{(\arabic*)}]
    \item $X$ equal to the punctured disc and $f(n) = \frac{n-2}{2}$;
    \item\label{case:torus} $X$ equal to the one-holed torus and $f(n) = \frac{n-3}{2}$;
    \item\label{case:telescope} $X$ equal to any telescope $\mathfrak{L}(\Sigma)$ and $f(n)= \frac{n-3}{2}$;
    \item\label{case:Cantor} $X$ equal to any binary tree surface $\mathfrak{B}(\Sigma)$ and $f(n) \equiv \infty$.
\end{enumerate}
\end{athm}

\begin{rem}
Case \ref{case:Cantor} of Theorem \ref{thm:hom-stab} is an ingredient in the proof of Theorem \ref{thm:disc-minus-Cantor}.
\end{rem}

\begin{proof}[Proof of Theorem \ref{thm:hom-stab}]
By Theorems \ref{thm:RWW} and \ref{thm:reduction-to-TC}, it suffices to prove that the simplicial complex $TC_n(A,X)$ is highly-connected. This connectivity result is then proved in Theorem \ref{thm-conn-fin}, Corollary \ref{thm-conn-lins-shifted} and Theorem \ref{thm-conn-btree-shifted}. Note that in case \ref{case:torus} one cannot directly apply Theorem \ref{thm:RWW} to get our stable range when $A$ has genus zero; in this case, we instead replace $A$ with $A' = A \natural T_0$ where $T_0$ is the one-holed torus. In cases \ref{case:telescope} and \ref{case:Cantor} we similarly replace $A$ with $A \natural \mathfrak{L}(\Sigma)$ or $A \natural \mathfrak{B}(\Sigma)$ respectively and then Theorem \ref{thm:RWW} implies that the $n$-th map in the sequence induces isomorphisms on homology up to degree $\frac{n-3}{2}$. One then notices that, for any surface $\Sigma$ without boundary, the binary tree surface $\mathfrak{B}(\Sigma)$ is homeomorphic to $\mathfrak{B}(\Sigma) \natural \mathfrak{B}(\Sigma)$. In particular, the groups $\Map(A \natural \mathfrak{B}(\Sigma)^{\natural n})$ are all isomorphic for $n\geq 1$ and the maps in the sequence \eqref{eq:stab-maps} are all identified in case \ref{case:Cantor}; it then follows that the stable range is infinite in this case, as claimed.
\end{proof}

\addtocontents{toc}{\SkipTocEntry}
\subsection*{Strategy.}

We prove Theorem \ref{thm:disc-minus-Cantor} using an adaptation of the concept of ``dissipated groups'', which we call \emph{homologically dissipated groups}, together with homological stability. The idea for the notion of homologically dissipated groups is inspired by an argument of Mather~\cite{Mather1971} and gives a criterion for the mapping class group of a space to be acyclic, the criterion being a certain kind of self-similarity and homological (auto-)stability (see also Remark \ref{rem:difference-in-hypotheses}). We then apply the machinery of \cite{RWW17} and adapt techniques of \cite{SW19} and \cite{SkipperWu2021} to prove the homological auto-stability hypothesis, which is item \ref{case:Cantor} of Theorem \ref{thm:hom-stab} with $A = \bD^2$. We use a number of different arguments to prove the other cases of Theorem \ref{thm:hom-stab}, in each case applying the machinery of \cite{RWW17} and studying the action of the mapping class group on appropriate tethered curve complexes on the surface. We expect that our methods can also be used to prove acyclicity of groups in other settings.

\addtocontents{toc}{\SkipTocEntry}
\subsection*{Outline.}

We begin in \S\ref{s:bg} with background on the classification of (second countable) surfaces, together with some homological stability tools recalled from \cite{RWW17,HatcherVogtmann2017,HW10} that we will need later. In \S\ref{s:Mather} we adapt an argument of Mather to give a sufficient criterion for acyclicity of mapping class groups whose key hypothesis is a \emph{homological auto-stability} property. We apply this to $\Map(\mathfrak{B}(\Sigma))$ in \S\ref{subsec:Mather-application}, completing the proof of Theorem \ref{thm:disc-minus-Cantor} modulo case \ref{case:Cantor} of Theorem \ref{thm:hom-stab}. In \S\ref{s:hs} we set up the appropriate pre-braided homogeneous category to apply the machinery of \cite{RWW17} to mapping class groups of infinite-type surfaces. The main theorem of \cite{RWW17} thus reduces Theorem \ref{thm:hom-stab} to proving high-connectivity for certain sequences of simplicial complexes, which we prove for each case of Theorem \ref{thm:hom-stab} in \S\ref{s:high-connectivity} (see Figure \ref{fig:structure-of-proof} for the outline of the proof).

\addtocontents{toc}{\SkipTocEntry}
\subsection*{Acknowledgements.} 
This work was initiated during the ``Big mapping class groups" reading seminar at Bielefeld in Spring 2021. We  thank the members of the reading group. MP was partially supported by a grant of the Romanian Ministry of Education and Research, CNCS - UEFISCDI, project number PN-III-P4-ID-PCE-2020-2798, within PNCDI III. XW is currently a member of LMNS and supported by a starter grant at Fudan University. He would like to thank Kai-Uwe Bux for his support when he was a postdoc at Bielefeld and many helpful discussions. He also thanks Javier Aramayona, Lvzhou Chen and Jonas Fleisig for discussions related to this project. After informing Lvzhou Chen of our results, he informed us recently of his on-going joint work with Danny Calegari and Nathalie Wahl on closely related topics.

\section{Background}\label{s:bg}

\subsection{Infinite-type surfaces}\label{s:inf-type}

All surfaces in this paper are assumed to be second countable, connected, orientable and to have compact (possibly empty) boundary. If the fundamental group of $S$ is finitely generated, we say that $S$ is of \emph{finite type}; otherwise, we say that $S$ is of \emph{infinite type}. Surfaces of finite type are classified by their genus, number of boundary components and number of punctures. For infinite-type surfaces, the corresponding classification was proven by von Ker\'ekj\'art\'o \cite{vKer23} and Richards \cite{Ri63}. To describe their classification, we first introduce some terminology. An \emph{end} of $S$ is an element of the set
\begin{equation}
\label{eq:inverse-limit}
\Ends(S) = \varprojlim \pi_0( S\smallsetminus K),
\end{equation}
where the inverse limit is taken over all compact sets $K \subset S$, directed with respect to inclusion. The \emph{Freudenthal compactification} of $S$ is the space
\[
\overline{S} = S \sqcup \Ends(S)
\]
equipped with the topology generated by the subsets $U \sqcup \{ e \in \Ends(S) \mid e<U \}$ for all open $U \subseteq S$, where $e<U$ means that there is a compact $K \subset S$ such that $U$ contains the component of $S \smallsetminus K$ hit by $e$ under the natural map $\Ends(S) \to \pi_0(S \smallsetminus K)$. In particular, taking the subspace topology, this equips $\Ends(S)$ with a topology, which coincides with the limit topology arising from the discrete topology on each of the terms defining the inverse limit \eqref{eq:inverse-limit}. With this topology, $\Ends(S)$ is homeomorphic to a closed subset of the Cantor set. An end $e \in \Ends(S)$ is called \emph{planar} if it has a neighbourhood (in the topology of $\overline{S}$) that embeds into the plane; otherwise it is {\em non-planar}. Denote the subspace of non-planar ends by $\Ends_{np}(S) \subseteq \Ends(S)$. This is a closed subset of $\Ends(S)$.

\begin{thm}[{\cite[\S 4.5]{Ri63}}]
Let $S_1,S_2$ be two surfaces with genera $g_1,g_2 \in \bN \cup \{\infty\}$ and with $b_1,b_2 \in \bN$ boundary components. Then $S_1$ is homeomorphic to $S_2$ if and only if $g_1 = g_2$, $b_1 = b_2$ and there is a homeomorphism of pairs
\[
(\Ends(S_1), \Ends_{np}(S_1)) \cong (\Ends(S_2), \Ends_{np}(S_2)).
\]
Conversely, given any $g \in \bN \cup \{\infty\}$, $b \in \bN$ and nested closed subsets of the Cantor set $X\subseteq Y$, where we assume that $g = \infty$ if and only if $X \neq \varnothing$, there is a surface $S$ of genus $g$ and with $b$ boundary components such that $(\Ends(S),\Ends_{np}(S)) \cong (Y,X)$.
\end{thm}

\subsection{End spaces of linear and binary tree surfaces.}
\label{ss:end-spaces}

Let $E$ be a compact topological space and denote by $E\omega$ the disjoint union of countably infinitely many copies of $E$. We will consider two compactifications of $E\omega$, its one-point compactification $(E\omega)^+$ and its \emph{Cantor compactification} $(E\omega)^\cC$, which we introduce below. These are part of a larger family of compactifications of $E\omega$, one for every non-empty compact metrisable space $X$. To describe these, we first recall the classification of metrisable compactifications of the natural numbers $\bN$ with the discrete topology.

\begin{prop}[{cf.~\cite{Lorch1982} or \cite[Propositions~2.1~and~2.3]{Tsankov2006}}]
\label{prop:compactifications-of-N}
Every non-empty compact metrisable space can be realised uniquely as the boundary of a metrisable compactification of the natural numbers $\bN$ with the discrete topology. Precisely, sending a metrisable compactification of $\bN$ to its boundary defines a one-to-one correspondence between metrisable compactifications of $\bN$ up to homeomorphism of pairs and non-empty compact metrisable spaces up to homeomorphism.
\end{prop}

\begin{defn}
\label{defn:X-compactification}
Let $E$ be a compact space and let $X$ be a non-empty compact metrisable space. Denote by $\widetilde{X}$ the unique metrisable compactification of $\bN$ with boundary $X$ (Proposition \ref{prop:compactifications-of-N}). The \emph{$X$-compactification} $(E\omega)^X$ of $E\omega$ is defined to be the quotient of the product $E \times \widetilde{X}$ by the equivalence relation that partitions it into $\{e\} \times \{n\}$ and $E \times \{x\}$ for $n \in \bN$, $x \in X$ and $e \in E$; in other words, each $E \times \{x\}$ is collapsed to a point. Its underlying set is therefore naturally in bijection with $E\omega \sqcup X = (E \times \bN) \sqcup X$.
\end{defn}

\begin{ex}
When $X$ is a single point, $\widetilde{X}$ is the one-point compactification of $\bN$ and $(E\omega)^X = (E\omega)^+$ is the one-point compactification of $E\omega$.
\end{ex}

\begin{defn}
\label{defn:Cantor-compactification}
In the special case when $X = \cC$ is the Cantor set, we call $(E\omega)^\cC$ the \emph{Cantor compactification} of $E\omega$.
\end{defn}

\begin{rem}
The Cantor compactification may be made more concrete by giving a concrete description of $\widetilde{\cC}$, the unique metrisable compactification of $\bN$ with boundary $\cC$. This may be realised as the subspace of $[0,1]$ given by the union of the middle-thirds Cantor set $\cC \subset [0,1]$ and the set of midpoints of the intervals forming the components of $[0,1] \smallsetminus \cC$. Even more concretely, this consists of all real numbers whose ternary expansion is of the form $0.w$ or $0.xy$, where $w$ is an infinite sequence of $0$s and $2$s, $x$ is a \emph{finite} sequence of $0$s and $2$s and $y$ is an infinite sequence of $1$s.
\end{rem}

Let $\cE$ denote the collection of topological spaces that are compact, separable and totally disconnected. The one-point and Cantor compactifications of $E\omega$ for $E \in \cE$ give us two operations
\[
\Upsilon^+ \text{ and } \Upsilon^\cC \colon \cE \longrightarrow \cE ,
\]
where $\Upsilon^+(E) = (E\omega)^+$ and $\Upsilon^\cC(E) = (E\omega)^\cC$.
These correspond to the operations $\mathrm{Gr}_\mathfrak{L}(-)$ and $\mathrm{Gr}_\mathfrak{B}(-)$ on surfaces, taking a connected surface $\Sigma$ to its associated \emph{linear surface} and \emph{binary tree surface} respectively:
\[
\begin{tikzcd}
\mathrm{Surfaces} \ar[d,swap,"{\Ends}"] \ar[rr,"{\mathrm{Gr}_\mathfrak{L}(-)}"] && \mathrm{Surfaces} \ar[d,"{\Ends}"] && \mathrm{Surfaces} \ar[d,swap,"{\Ends}"] \ar[rr,"{\mathrm{Gr}_\mathfrak{B}(-)}"] && \mathrm{Surfaces} \ar[d,"{\Ends}"] \\
\cE \ar[rr,"{\Upsilon^+}"] && \cE && \cE \ar[rr,"{\Upsilon^\cC}"] && \cE
\end{tikzcd}
\]

\begin{rem}
\label{rmk:Cantor-compactifications}
The spaces $\Upsilon^\cC([0,\omega^\alpha])$ are non-homeomorphic for distinct countable ordinals $\alpha$. To see this, we extract the ordinal $\alpha$ as a topological invariant of $\Upsilon^\cC([0,\omega^\alpha])$. First, removing the condensation points (those points all of whose neighbourhoods are uncountable) of $\Upsilon^\cC([0,\omega^\alpha])$ results in the disjoint union of countably infinitely many copies of $[0,\omega^\alpha]$. The Cantor-Bendixson filtration of this disjoint union is simply the disjoint union of countably infinitely many copies of the Cantor-Bendixson filtration of $[0,\omega^\alpha]$, which has precisely $\alpha + 1$ steps. Thus $\alpha$ is a topological invariant of the collection of spaces of the form $\Upsilon^\cC([0,\omega^\alpha])$, so they are non-homeomorphic for distinct $\alpha$.

Since the homeomorphism type of $\Ends(S)$ is an invariant of a surface $S$, this implies that there are uncountably many different surfaces of the form $\mathfrak{B}(\Sigma)$, as pointed out in Remark \ref{rmk:uncountably-many-examples}. Namely, the surface $\mathfrak{B}(\bS^2 \smallsetminus [0,\omega^\alpha])$ has end space $\Upsilon^\cC([0,\omega^\alpha])$ and thus the surfaces $\mathfrak{B}(\bS^2 \smallsetminus [0,\omega^\alpha])$ are pairwise non-homeomorphic as $\alpha$ ranges over the (uncountably many) countable ordinals.
\end{rem}

\begin{rem}
The operations $\Upsilon^+$ and $\Upsilon^\cC$ are clearly functorial and send embeddings to embeddings. Using the topological characterisation of the Cantor set $\cC$, one may see that
\[
\Upsilon^+(\cC) \cong \cC \qquad\text{and}\qquad \Upsilon^\cC(\cC) \cong \cC .
\]
Thus, fixing a choice of such homeomorphisms, an embedding $E \hookrightarrow \cC$ determines embeddings $\Upsilon^+(E) \hookrightarrow \cC$ and $\Upsilon^\cC(E) \hookrightarrow \cC$.
\end{rem}

\subsection{Homogeneous categories and homological stability}\label{s:homogeneous-categories}

In this subsection, we review the basic definitions and results of homogeneous categories and refer the reader to \cite{RWW17} for more details.

\begin{defn}[{\cite[Definition 1.3]{RWW17}}]
A monoidal category $(\mathcal{C}, \oplus, 0)$ is called \emph{homogeneous} if $0$ is initial in $\mathcal{C}$ and if the following two properties hold.
\begin{inlinecond}{\textbf{H1}}
$\Hom(A,B)$ is a transitive $\Aut(B)$-set under post-composition.
\end{inlinecond}
\begin{inlinecond}{\textbf{H2}}
The map $\Aut(A) \rightarrow \Aut(A \oplus B)$ taking $f$ to $f \oplus \id_B$ is injective with image
\[
\Fix(B) := \{ \phi\in \Aut(A\oplus B) \mid \phi\circ (\imath_A\oplus \id_B) = \imath_A \oplus \id_B\}
\]
where $\imath_A\colon 0 \to A$ is the unique map.
\end{inlinecond}
\end{defn}

\begin{defn}[{\cite[Definition 1.5]{RWW17}}]
Let $(\mathcal{C}, \oplus, 0)$ be a monoidal category with $0$ initial. We say that $\cC$ is \emph{prebraided} if its underlying groupoid is braided and for each pair of objects $A$ and $B$ in $\cC$, the groupoid braiding $b_{A,B} \colon A \oplus B \to B\oplus A$ satisfies
\[
b_{A,B} \circ (\id_A \oplus \imath_B) = \imath_B\oplus \id_A \colon A \longrightarrow B\oplus A.
\]
\end{defn}

The homogeneous category we work with later will arise from a braided monoidal groupoid using Quillen's bracket construction \cite[p.~219]{Grayson1976}. Let $(\cG,\oplus,0)$ be a monoidal groupoid. One defines a new category $U\cG$ as follows: $U\cG$ has the same objects as $\cG$ and a morphism in $U\cG$ from $A$ to $B$ is an equivalence class of pairs $(X,f)$ where $X$ is an object of $\cG$ and $f \colon X \oplus A \to B$ is a morphism in $\cG$, and where $(X,f) \sim (X',f')$ if there exists an isomorphism $g \colon X\to X'$ in $\cG$ making the diagram 
\[
\begin{tikzcd}
X \oplus A \ar[r,"f"] \ar[d,swap,"{g \oplus \id_A}"] & B \\
X' \oplus A \ar[ur,swap,"{f'}"]
\end{tikzcd}
\]
commute. We write $[X,f]$ for such an equivalence class.

\begin{defn}
For a pair of objects $(A,X)$ in a monoidcal category $(\cG,\oplus,0)$ we say that $\cG$ satisfies \emph{cancellation} if, for all objects $A,B,C\in \cG$, $A\oplus C\cong B\oplus C$ implies $A\cong B$.
\end{defn}

\begin{thm}[{\cite[Proposition 1.8 and Theorem 1.10]{RWW17}}]
\label{thm-grpd-hgca}
Let $(\cG,\oplus,0)$ be a braided monoidal groupoid with no zero divisors. Then $U\cG$ is a pre-braided monoidal category. If in addition
\begin{enumerate}
    \item\label{cond-cancellation} $\cG$ satisfies cancellation,
    \item\label{cond-injectivity} $\Aut_\cG(A) \to \Aut_\cG(A\oplus B), f \mapsto f\oplus \id_B$ is injective for all $A,B$ in $\cG$,
\end{enumerate}
then $U\cG$ is a homogeneous category.
\end{thm}

\begin{defn}[{\cite[Definition 2.1]{RWW17}}]
\label{defn-W_n(A,X)}
Let $(\cC,\oplus,0)$ be a monoidal category with $0$ initial and let $(A,X)$ be a pair of objects of $\cC$. Define $W_n(A,X)_\bullet$ to be the semi-simplicial set with set of $p$-simplices
\[
W_n(A,X)_p := \Hom_\cC(X^{\oplus p+1}, A\oplus X^{\oplus n})
\]
and with face maps
\[
d_i \colon \Hom_\cC(X^{\oplus p+1}, A \oplus X^{\oplus n}) \longrightarrow \Hom_\cC (X^{\oplus p}, A \oplus X^{\oplus n})
\]
defined by precomposing with $\id_{X^{\oplus i}} \oplus \imath_X \oplus \id_{X^{\oplus p-i}}$.
\end{defn}

For connectivity questions, one can often replace this semi-simplicial set by its associated simplicial complex. (This will be the case in our setting; see Proposition \ref{prop:from-W-to-S}.)

\begin{defn}[{\cite[Definition 2.8]{RWW17}}]
\label{defn:Sn(X,A)}
Let $(A,X)$ be a pair of objects of a homogeneous category $(\mathcal{C}, \oplus,0)$ and $n \geq 1$. Define $S_n(A,X)$ to be the simplicial complex whose vertices are morphisms $f\colon X \to A\oplus X^{\oplus n}$ and whose $p$-simplices are $(p+1)$-tuples $(f_0, \ldots, f_p)$ of such morphisms such that there exists a morphism $f\colon X^{\oplus p+1} \to A\oplus X^{\oplus n}$ with $\{f \circ i_0, \ldots, f \circ i_p \} = \{f_0,\ldots,f_p\}$ as unordered sets, where
\[
i_j = \imath_{X^{\oplus j}} \oplus \id_X \oplus \imath_{X^{\oplus p-j}}  \colon X = 0 \oplus X \oplus 0 \longrightarrow X^{\oplus p+1}.
\]
\end{defn}

\subsection{High-connectivity}
\label{s:high-conn-bg}

We now review some of the tools that we will use to prove our connectivity results. More details may be found in \cite[\S 2]{HatcherVogtmann2017} and \cite[\S 3]{HW10}.

\subsubsection{Bad simplices arguments}
\label{sub-bad-sim}

Let $X$ be a simplicial complex and $Y \subseteq X$ a subcomplex. We want to relate the $n$-connectedness of $Y$ to the $n$-connectedness of $X$ via an argument known as a ``bad simplices argument''; see for example \cite[\S 2.1]{HatcherVogtmann2017} for more details. One first chooses a set of simplices in $X \smallsetminus Y$, called \emph{bad simplices}, satisfying the following two conditions:

\begin{enumerate}
\item\label{bad-sim-1} Any simplex with no bad faces lies in $Y$. Here a ``face'' of a simplex $\sigma$ means the simplex spanned by any non-empty subset of its vertices, including $\sigma$ itself.
\item\label{bad-sim-2} If two faces of a simplex of $X$ are bad, then their join (which is also a simplex of $X$) is also bad.
\end{enumerate}

Simplices that have no bad faces are called \emph{good simplices}. In this terminology, condition \eqref{bad-sim-1} says that all good simplices lie in $Y$. The faces of a bad simplex may be good, bad or neither. There is also a relative notion for simplices in the link of a bad simplex $\sigma$, as follows: a simplex $\tau$ in $\Lk(\sigma)$ is called \emph{good for $\sigma$} if any bad face of $\tau \ast \sigma$ is contained in $\sigma$. The simplices that are good for $\sigma$ form a subcomplex of $\Lk(\sigma)$ denoted by $G_{\sigma}$; this is called the \emph{good link} of $\sigma$.

\begin{prop}[{\cite[Proposition 2.1]{HatcherVogtmann2017}}]
\label{prop-bad-sim}
Let $Y$ be a subcomplex of a simplicial complex $X$, and suppose that we have chosen a collection of bad simplices satisfiying conditions \eqref{bad-sim-1} and \eqref{bad-sim-2} above. Fix $n\geq 0$ and suppose that $G_\sigma$ is $(n-\text{dim}(\sigma)-1)$-connected for each bad simplex~$\sigma$. Then the pair $(X,Y)$ is $n$-connected, i.e., $\pi_i(X,Y) = 0$ for all $i\leq n$. 
\end{prop}

An immediate consequence is:

\begin{thm}[{\cite[Corollary 2.2]{HatcherVogtmann2017}}]
\label{thm--bad-sim}
Let $Y$ be a subcomplex of a simplicial complex $X$, and suppose that we have chosen a collection of bad simplices satisfiying conditions \eqref{bad-sim-1} and \eqref{bad-sim-2} above. Then:
\begin{enumerate}
\item If $X$ is $n$-connected and $G_\sigma$ is $(n-\text{dim}(\sigma))$-connected for each bad simplex $\sigma$, then $Y$ is $n$-connected.
\item If $Y$ is $n$-connected and $G_\sigma$ is $(n-\text{dim}(\sigma)-1)$-connected for each bad simplex $\sigma$, then $X$ is $n$-connected.
\end{enumerate}
\end{thm}

\subsubsection{Complete joins}
\label{subsec-com-join}

The notion of a \emph{complete join} is another useful tool, introduced by Hatcher and Wahl in \cite[\S 3]{HW10}, for proving high-connectivity results. We recall the key definitions and results here.

\begin{defn}
\label{defn:complete-join}
A surjective map of simplicial complexes $\pi \colon Y \to X$ is called a \emph{complete join} if it satisfies the following properties:
\begin{enumerate}
\item\label{complete-join-injective} $\pi$ is injective on individual simplices.
\item\label{complete-join-lifting} For each $p$-simplex $\sigma = \langle v_0,\ldots,v_p \rangle$ of $X$, the subcomplex $Y(\sigma)$ of $Y$ consisting of all $p$-simplices that project under $\pi$ to $\sigma$ is equal to the join $\pi^{-1}(v_0) \ast \pi^{-1}(v_1) \ast \cdots \ast \pi^{-1}(v_p)$, where each $\pi^{-1}(v_i)$ is considered as a $0$-dimensional subcomplex of $Y$.
\end{enumerate}
Condition \eqref{complete-join-lifting} in particular asserts the \emph{existence} of the join $\pi^{-1}(v_0) \ast \pi^{-1}(v_1) \ast \cdots \ast \pi^{-1}(v_p)$ in $Y$. Another way of rephrasing condition \eqref{complete-join-lifting} is the statement that lifts of $\sigma$ to $p$-simplices of $Y$ are in one-to-one correspondence with $(p+1)$-tuples consisting of a choice of lift of $v_i$ to a vertex of $Y$ for each $0\leq i\leq p$. More succinctly: simplices of $X$ may be lifted to $Y$ by arbitrarily and independently choosing lifts of each vertex.
\end{defn}

\begin{defn}
A simplicial complex $X$ is called \emph{weakly Cohen-Macaulay} (sometimes shortened to \emph{wCM}) of dimension $n$ if $X$ is $(n-1)$-connected and the link of each $p$-simplex of $X$ is $(n-p-2)$-connected.
\end{defn}

The key result that we shall use for complete joins is the following.

\begin{prop}[{\cite[Proposition 3.5]{HW10}}]
\label{prop-join-conn}
If $Y$ is a complete join over a wCM complex $X$ of dimension $n$, then $Y$ is also wCM of dimension $n$.
\end{prop}

\begin{rem}
\label{rem-cjoin}
If $\pi \colon Y\to X$ is a complete join, then $X$ is a retract of $Y$. In fact, we may define a simplicial map $s \colon X\to Y$ such that $\pi\circ s = \id_X$ by sending a vertex $v\in X$ to any vertex in $\pi^{-1}(v)$ and then extending it to simplices. The fact that $s$ can be extended to simplices is guaranteed by the condition that $\pi$ is a complete join. In particular we may also conclude that if $Y$ is $n$-connected, so is $X$. Proposition \ref{prop-join-conn} is a kind of converse to this.
\end{rem}

\section{Abstracting Mather's infinite iteration argument}
\label{s:Mather}

The core of our argument, reducing the proof of acyclicity of a given mapping class group to proving a kind of homological auto-stability for it, is an adaptation of an infinite iteration (or ``suspension'') argument due to Mather \cite{Mather1971}. We abstract and axiomatise the version that we will need in \S\ref{subsec:Mather-axiomatic} (Proposition \ref{prop:Mather-axiomatic}) and apply it to our setting in \S\ref{subsec:Mather-application}.

\subsection{The general argument}
\label{subsec:Mather-axiomatic}

\begin{defn}
Let $T$ be a topological space and $S \subseteq T$ a closed subspace. Write $\partial_T(S)$ for the topological boundary of $S$ in $T$, that is, $\partial_T(S) = S \cap (\overline{T \smallsetminus S})$.
\end{defn}

We reserve the simpler notation $\partial$ to denote the boundary of a topological manifold. In our situation, $S$ will be a topological manifold but $\partial_T(S)$ will not necessarily coincide with $\partial S$. To illustrate the difference between topological boundary and boundaries of manifolds, we mention the following examples.

\begin{ex}
{~}

\begin{enumerate}
\item Let $T=\bR^n$ and $S \subset \bR^n$ be a compact $n$-dimensional submanifold. Then $\partial_T(S) = \partial S$.
\item Let $T = \bR^n$ and $S \subset \bR^n$ a closed submanifold of positive codimension. Then $\partial_T(S) = S$.
\item Let  $S = [0,1]\times [0,1]\subset [0,2] \times [0,1] = T$. Then $\partial_T(S) = \{1\} \times  [0,1] \subsetneq \partial S$.
\end{enumerate}
\end{ex}

\begin{notation}
For a pair of spaces $A \subseteq X$, write $\mathrm{Homeo}_{A}(X)$ for the topological group of homeomorphisms of $X$ that restrict to the identity on $A \subseteq X$, with the compact-open topology.
\end{notation}

Let $B,S \subseteq T$ be  closed subspaces with $B \cap S = \varnothing$. There is a continuous group homomorphism
\begin{equation}
\label{eq:extend-by-identity-homeo}
\mathrm{Homeo}_{\partial_T(S)}(S) \longrightarrow \mathrm{Homeo}_B(T)
\end{equation}
given by extending homeomorphisms by the identity on $T\smallsetminus S$. Precisely, given a homeomorphism $\varphi$ of $S$ that restricts to the identity on $\partial_T(S)$, let $\widehat{\varphi}$ be the homeomorphism of $T$ defined by $\varphi$ on $S$ and by the identity on $T \smallsetminus S$. The fact that $\widehat{\varphi}$ is continuous uses the assumption that $\varphi$ restricts to the identity on $\partial_T(S)$.

The first step towards our version of Mather's ``suspension'' argument is the following lemma, which constructs self-homeomorphisms of $T$ by \emph{infinite iteration}.

\begin{lem}[Infinite iteration lemma.]
\label{lem:Mather-preparation-lemma}
Suppose that we are given $\varphi \in \mathrm{Homeo}_B(T)$ such that
\begin{itemizeb}
\item $\varphi^k(S)$ is disjoint from $S$ for all $k\geq 1$ and
\item the union $\bigcup_{k\geq i} \varphi^k(S)$ is a closed subset of $T$ for each $i\geq 0$.
\end{itemizeb}
Let $h \in \mathrm{Homeo}_{\partial_T(S)}(S)$. Then the assignment
\begin{equation}
\label{eq:infinite-composition-of-conjugations}
t \longmapsto \begin{cases}
\varphi^k(h(\varphi^{-k}(t))) & t \in \varphi^k(S) \text{ for } k \geq 0 \\
t & \text{otherwise}
\end{cases}
\end{equation}
gives a well-defined element of $\mathrm{Homeo}_B(T)$.
\end{lem}

\begin{notation}
Under the conditions of the above lemma, we will denote the element \eqref{eq:infinite-composition-of-conjugations} of $\mathrm{Homeo}_B(T)$ by
\[
h^{\varphi^{\infty}} \colon T \longrightarrow T.
\]
\end{notation}

\begin{rem}
The second hypothesis on $\varphi$ in the statement of Lemma \ref{lem:Mather-preparation-lemma} is equivalent to the a priori stronger-looking statement that $\bigcup_{k \in N} \varphi^k(S)$ is a closed subset of $T$ for any subset $N \subseteq \{0,1,2,\ldots\}$. Intuitively, this may be thought of as saying that the sequence of closed subsets $\varphi^k(S)$ ``converges to an end'' of the space $T$.
\end{rem}

\begin{proof}[Proof of Lemma \ref{lem:Mather-preparation-lemma}]
The first hypothesis on $\varphi$ ensures that the subsets $\varphi^k(S) \subseteq T$ for $k\geq 0$ are pairwise disjoint. Thus \eqref{eq:infinite-composition-of-conjugations} gives a well-defined bijection $T \to T$. We next check that it is continuous. To do this, it will suffice to check that it is continuous when restricted to each piece of a finite covering of $T$ by closed subsets. We will use the covering of $T$ by the two subsets
\[
\textstyle
\bigcup_{k\geq 0} \varphi^k(S) \qquad\text{and}\qquad T \smallsetminus \bigl( \bigcup_{k\geq 0} \varphi^k(\mathrm{int}(S)) \bigr) .
\]
These clearly cover $T$, the first one is closed by the second hypothesis on $\varphi$ (with $i=0$) and the second is the complement of a union of open subsets, hence closed. The map \eqref{eq:infinite-composition-of-conjugations} is \emph{constant} when restricted to the second closed subset (this uses the fact that $h$ restricts to the identity on $\partial_T(S)$); in particular, its restriction to the second closed subset is continuous. It thus suffices to show that \eqref{eq:infinite-composition-of-conjugations} is continuous when restricted to $\bigcup_{k\geq 0} \varphi^k(S)$. The second hypothesis on $\varphi$ (for all $i\geq 0$) implies that each $\varphi^k(S)$ is an open subset of the union $\bigcup_{k\geq 0} \varphi^k(S)$ in the subspace topology induced from $T$. Thus, to show that \eqref{eq:infinite-composition-of-conjugations} is continuous when restricted to $\bigcup_{k\geq 0} \varphi^k(S)$, it will suffice to show that it is continuous when restricted to $\varphi^k(S)$ for each fixed $k\geq 0$. But this is true by construction. Thus we have shown that \eqref{eq:infinite-composition-of-conjugations} is a continuous bijection $T \to T$.

To see that the inverse of \eqref{eq:infinite-composition-of-conjugations} is continuous, we simply replace $h$ with $h^{-1}$ in the formula \eqref{eq:infinite-composition-of-conjugations}: this clearly gives the inverse of \eqref{eq:infinite-composition-of-conjugations}, and it is continuous by the above reasoning with $h^{-1}$ in place of $h$. Thus we have shown that $\eqref{eq:infinite-composition-of-conjugations} \in \mathrm{Homeo}(T)$.

Finally, we must show that \eqref{eq:infinite-composition-of-conjugations} restricts to the identity on $B$. For this, it will suffice to show that $B$ is disjoint from $\varphi^k(S)$ for all $k\geq 0$. This follows by induction on $k$: the base case that $B \cap S = \varnothing$ is true by assumption, and the inductive step follows from the fact that the self-homeomorphism $\varphi \colon T \to T$ fixes $B$.
\end{proof}

The continuous group homomorphism \eqref{eq:extend-by-identity-homeo} induces a group homomorphism of the corresponding mapping class groups:
\begin{equation*}
\iota \colon \pi_0 \bigl( \mathrm{Homeo}_{\partial_T(S)}(S) \bigr) \longrightarrow \pi_0 \bigl( \mathrm{Homeo}_B(T) \bigr)
\end{equation*}
Let us choose subgroups
\[
H \subseteq \pi_0 \bigl( \mathrm{Homeo}_{\partial_T(S)}(S) \bigr) \qquad\text{and}\qquad G \subseteq \pi_0 \bigl( \mathrm{Homeo}_B(T) \bigr)
\]
such that $\iota(H) \subseteq G$, so that we have a restricted homomorphism
\begin{equation*}
\iota|_H \colon H \longrightarrow G.
\end{equation*}

\begin{prop}[The infinite iteration argument for mapping class groups.]
\label{prop:Mather-axiomatic}
In the above setup, suppose that
\begin{itemize}
\item the homomorphism $\iota|_H$ induces isomorphisms on homology in all degrees;
\item there is an element $\varphi \in \mathrm{Homeo}_B(T)$ satisfying the two hypotheses of Lemma \ref{lem:Mather-preparation-lemma}, such that $[\varphi] \in G$ and, for each $h \in \mathrm{Homeo}_{\partial_T(S)}(S)$ with $[h] \in H$, we have $[h^{\varphi^{\infty}}] \in G$.
\end{itemize}
Then $G$ is acyclic.
\end{prop}

\begin{rem}
The first condition in the hypotheses of Proposition \ref{prop:Mather-axiomatic} is to be thought of as a \emph{homological auto-stability} property of the group $G$.
\end{rem}

\begin{rem}
\label{rmk:special-case-of-Mather}
In practice, we will always apply this proposition in the special case where we have $H = \pi_0(\Homeo_{\partial_T(S)}(S))$ and $G = \pi_0(\Homeo_B(T))$, so the second hypothesis simplifies to the statement that there exists an element $\varphi \in \mathrm{Homeo}_B(T)$ satisfying the two hypotheses of Lemma \ref{lem:Mather-preparation-lemma}.
\end{rem}

\begin{defn}
If a group $G$ admits the structure described above, then we say that it is \emph{homologically dissipated} and that $[\varphi] \in G$ is a \emph{homological dissipator} for $G$. In this terminology, Proposition \ref{prop:Mather-axiomatic} is the statement that \emph{homologically dissipated groups are acyclic}. This should be compared to \emph{binate groups} \cite{Berrick1989} and \emph{dissipated groups} \cite{Berrick2002}; see Remark \ref{rem:difference-in-hypotheses} below.
\end{defn}

\begin{proof}[Proof of Proposition \ref{prop:Mather-axiomatic}]
This is an almost verbatim adaptation of \cite{Mather1971}. However, we give a complete proof, since this technical result is central to our arguments in this paper.

We first note that the construction of Lemma \ref{lem:Mather-preparation-lemma} gives a well-defined function
\[
h \mapsto h^{\varphi^{\infty}} \colon \mathrm{Homeo}_{\partial_T(S)}(S) \longrightarrow \mathrm{Homeo}_B(T).
\]
It is not hard to check that this is continuous in the compact-open topology, and it is clearly a homomorphism. It thus induces a homomorphism
\[
\pi_0 \bigl( \mathrm{Homeo}_{\partial_T(S)}(S) \bigr) \longrightarrow \pi_0 \bigl( \mathrm{Homeo}_B(T) \bigr) ,
\]
which restricts, by our hypotheses, to a homomorphism
\[
j \colon H \longrightarrow G.
\]
We have also assumed that $[\varphi] \in G$, so we have an inner automorphism $[\varphi] \cdot - \cdot [\varphi]^{-1} \colon G \to G$. Define
\[
l \colon H \longrightarrow G
\]
to be the composition $([\varphi] \cdot - \cdot [\varphi]^{-1}) \circ j$. Notice that $l([h])$ is represented by the homeomorphism of $T$ given by the formula
\begin{equation*}
t \longmapsto \begin{cases}
\varphi^k(h(\varphi^{-k}(t))) & t \in \varphi^k(S) \text{ for } k \geq 1 \\
t & \text{otherwise}
\end{cases}
\end{equation*}
(notice that this differs from $\iota|_H$ only on $S \subseteq T$). Finally, we also consider the homomorphism
\[
i = \iota|_H \colon H \longrightarrow G;
\]
recall that this extends a homeomorphism of $S$ by the identity on $T \smallsetminus S$. Since $i([h])$ and $l([h])$ admit representatives with disjoint support (contained in $S$ and in $\bigcup_{k\geq 1} \varphi^k(S)$ respectively), we see that any element of $i(H)$ commutes with any element of $l(H)$. This means that the function
\[
\eta \colon H \times H \longrightarrow G
\]
defined by $\eta(h_0 , h_1) = i(h_0) l(h_1)$ is a homomorphism. If we denote by $\Delta \colon H \to H \times H$ the diagonal homomorphism, then by comparing the formulas for $j$ and for $l$ we see that $\eta \circ \Delta = j$.

We now prove by induction on $r \geq 0$ that $\widetilde{H}_s(G) = 0$ for all $s\leq r$. This is trivial for $r=0$. For the inductive step, we assume that $\widetilde{H}_s(G) = 0$ for all $s\leq r-1$. By the first hypothesis, this implies that $\widetilde{H}_s(H) = 0$ for all $s\leq r-1$. By the K{\"u}nneth formula, this implies that we have $H_r(H \times H) \cong H_r(H) \oplus H_r(H)$ and that for any $\alpha \in H_r(H)$, the class $\Delta_*(\alpha)$ corresponds to $\alpha \oplus \alpha$ under this isomorphism. We thus have
\[
j_*(\alpha) = \eta_*(\Delta_*(\alpha)) = \eta_*(\alpha \oplus \alpha) = i_*(\alpha) + l_*(\alpha).
\]
But $j$ and $l$ differ by an inner automorphism of $G$, which acts trivially on homology, so we also have $j_*(\alpha) = l_*(\alpha)$. Putting this together with the above formula, we deduce that $i_*(\alpha) = 0$. By the first hypothesis, $i_*$ is injective, so $\alpha = 0$. Thus we have shown that $H_r(H) = 0$. By the first hypothesis again, this means that also $H_r(G) = 0$, which finishes the inductive step.
\end{proof}

\begin{rem}
\label{rem:difference-in-hypotheses}
Our notion of \emph{homologically dissipated groups} is inspired by the classical notion of \emph{dissipated groups}~\cite{Berrick2002}. Roughly, a group $G$ acting on a set $X$ is \emph{dissipated} if it admits an exhaustion by subgroups $G_i$ defined by restricting the support of the action, and each $G_i$ admits a ``dissipator'' $\rho_i \in G$ (an element satisfying similar properties to the element \eqref{eq:infinite-composition-of-conjugations} provided by Lemma \ref{lem:Mather-preparation-lemma}). Dissipated groups are \emph{binate} ($\equiv$ \emph{pseudo-mitotic}) by \cite{Berrick2002}, which in turn are acyclic by \cite{Berrick1989} or \cite{Varadarajan1985}.

This notion cannot directly be applied to mapping class groups, since they do not naturally act on any set; instead, they act \emph{up to homotopy} on the underlying surface. However, one may adapt the notion of dissipated groups to an analogous notion of ``homotopy-dissipated groups'' that applies to mapping class groups and still implies acyclicity. One could then try to prove Theorem \ref{thm:disc-minus-Cantor} using the notion of homotopy-dissipated groups together with a ``fragmentation'' argument introduced by Segal \cite{Segal1978} and used in \cite{dlHM83, SergiescuTsuboi1994, SergiescuTsuboi1996, CC21}. However, such a strategy cannot work: the first step of this fragmentation argument in the case of $S = \bD^2 \smallsetminus \cC$ would be to prove that the subgroup $\Map(\bD^2 \smallsetminus \cC)_{\{0\}}$ of mapping classes that are the identity in a neighbourhood of $0 \in \cC$ is homotopy-dissipated. But this group surjects onto the mapping class group of the cylinder, which is $\bZ$ (this surjection is similar to the one described in \cite[paragraph before Lemma A.6]{CC21}), so it is not acyclic and therefore cannot be homotopy-dissipated.

Thus we instead work with the notion of \emph{homologically dissipated groups} introduced above -- which in turn leads to the use of homological stability arguments. The key difference between the notions of \emph{(homotopy-)dissipated groups} and \emph{homologically dissipated groups} is the following. In the former, $G$ admits an exhaustion by subgroups $G_i$ defined by restricting the support of homeomorphisms, and \emph{each} $G_i$ admits a dissipator $[\rho_i] \in G$. In the latter, $G$ has just a single subgroup $H$ defined by restricting the support of homeomorphisms, which admits a dissipator $[\varphi] \in G$ and whose inclusion into $G$ induces isomorphisms on homology.\footnote{In fact $H$ does not have to be a subgroup; we just require a homomorphism $H \to G$.} In other words, the condition that $H \subset G$ is a homology isomorphism replaces the requirement that we have a whole family of subgroups (equipped with dissipators) limiting to $G$.
\end{rem}

\subsection{Application to binary tree surfaces}
\label{subsec:Mather-application}

We now apply Proposition \ref{prop:Mather-axiomatic} to reduce the proof of acyclicity of $\Map(\mathfrak{B}(\Sigma))$ to a homological auto-stability result, which will be part of the results of the next section. Consider the boundary connected sum $\mathfrak{B}(\Sigma) \natural \mathfrak{B}(\Sigma)$ of two copies of $\mathfrak{B}(\Sigma)$. Extending homeomorphisms of one copy of $\mathfrak{B}(\Sigma)$ by the identity on the other copy, we have a homomorphism
\begin{equation}
\label{eq:stabilisation-e}
e \colon \Map(\mathfrak{B}(\Sigma)) \longrightarrow \Map(\mathfrak{B}(\Sigma) \natural \mathfrak{B}(\Sigma)).
\end{equation}

\begin{prop}
\label{prop:acyclicity-assuming-stability}
Suppose that $e$ is a homology isomorphism. Then the group $\Map(\mathfrak{B}(\Sigma))$ is acyclic.
\end{prop}

This reduces the proof of Theorem \ref{thm:disc-minus-Cantor} to proving that $e$ is a homology isomorphism, which we will do in the following two sections using the tools of homological stability.

\begin{figure}[tb]
    \centering
    \includegraphics{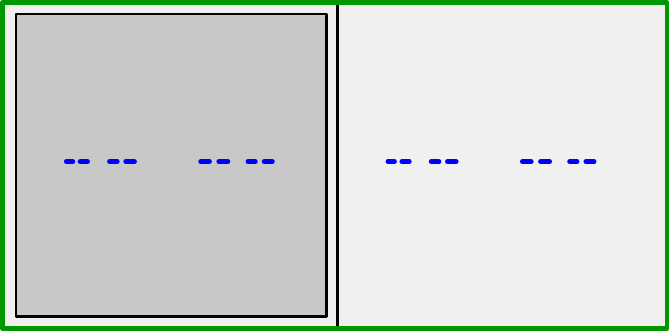}
    \caption{The surface $D_\cC$, which is the disc minus the Cantor set (drawn in blue here). The subsurface $S$ is drawn in darker grey and the boundary is green.}
    \label{fig:Disc-minus-Cantor-1}
    \vspace{2ex}
    \includegraphics[scale=0.37]{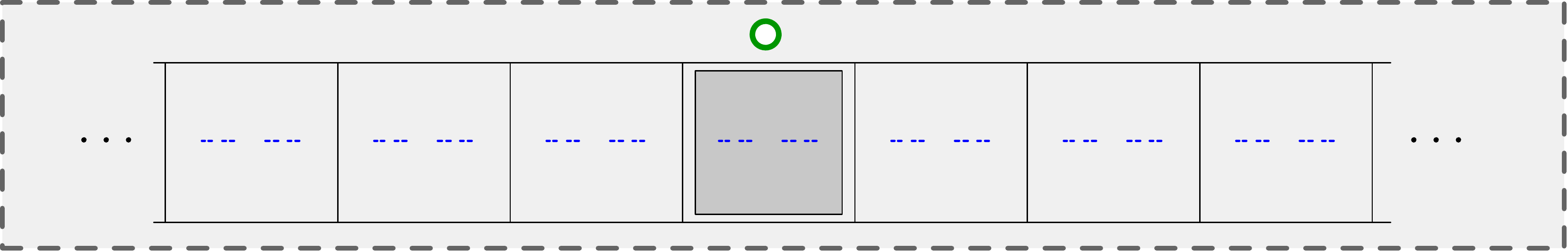}
    \caption{Another viewpoint of $D_\cC$ and the subsurface $S \subseteq D_\cC$. The outer dashed rectangle denotes the plane and the small green circle is $\partial D_\cC$; the inside of this circle is the \emph{outside} of $D_\cC$; we have turned the surface ``inside out'' compared to Figure \ref{fig:Disc-minus-Cantor-1} above and have also decomposed the Cantor set as the $1$-point compactification of a disjoint union of countably infinitely many copies of the Cantor set. Although the bi-infinite sequence of copies of $\cC$ appears to converge to two points at infinity in the picture, we emphasise that there is just a single end at infinity, corresponding to the unique end of the plane.}
    \label{fig:Disc-minus-Cantor-2}
\end{figure}

\begin{proof}
We give full details for the surface $D_\cC = \mathfrak{B}(\bS^2)$ and then indicate how to modify the argument for $\mathfrak{B}(\Sigma)$ for an arbitrary $\Sigma$.

Figures \ref{fig:Disc-minus-Cantor-1} and \ref{fig:Disc-minus-Cantor-2} show two depictions of the surface $D_\cC$, where the removed Cantor set is drawn in blue. Figure \ref{fig:Disc-minus-Cantor-2} should be read as follows: the outer dashed rectangle represents the plane and we remove the interior of a small disc (in the top middle of the figure) to obtain a closed $2$-disc; we then remove the bi-infinite sequence of copies of the Cantor set $\cC$ depicted in the figure in blue. The space of ends is thus the $1$-point compactification $(\cC\omega)^+$ of a countably infinite disjoint union $\cC\omega = \bigsqcup_\omega \cC$ of copies of $\cC$, which is again homeomorphic to the Cantor set. In each of the two figures we have also drawn a subsurface $S$, contained in the interior of $D_\cC$, which is again homeomorphic to $D_\cC$. We emphasise that Figures \ref{fig:Disc-minus-Cantor-1} and \ref{fig:Disc-minus-Cantor-2} are two illustrations of the \emph{same} pair of surfaces $(D_\cC,S)$.

The inclusion $D_\cC \hookrightarrow D_\cC \natural D_\cC$ is isotopy equivalent to the inclusion $S \hookrightarrow D_\cC$ depicted in Figure \ref{fig:Disc-minus-Cantor-1}, so the assumption of the proposition says that this inclusion is a homology isomorphism. Setting $T = D_\cC$, $B = \partial D_\cC$ and $S$ to be the subsurface depicted in grey, Proposition \ref{prop:Mather-axiomatic} implies that it suffices to find a homeomorphism $\varphi \in \Homeo_\partial(D_\cC)$ satisfying the two conditions of that proposition.

To construct $\varphi$ we consider Figure \ref{fig:Disc-minus-Cantor-2}. Define $\varphi$ to be the identity outside of the horizontal strip. On a slightly smaller horizontal strip that contains $S$, define $\varphi$ by translation one unit to the right, and extend this continuously in the remaining small neighbourhood of the top and bottom boundaries of the strip. The iterated images of $S$ under $\varphi$ are clearly disjoint, so the first condition is satisfied. Moreover, the union of any collection of the iterated images of $S$ under $\varphi$ is a closed subset of $D_\cC$ (an essential point here is that the point at infinity is not part of the surface), so the second condition is satisfied. Thus Proposition \ref{prop:Mather-axiomatic} implies that $\Map(D_\cC)$ is acyclic.

Finally, to generalise the argument to $\mathfrak{B}(\Sigma)$ for any surface $\Sigma$ in place of $\Sigma = \bS^2$, we modify Figures \ref{fig:Disc-minus-Cantor-1} and \ref{fig:Disc-minus-Cantor-2} appropriately, by taking infinitely many connected sums with copies of $\Sigma$ along a collection of pairwise disjoint discs converging to all ends of the surface; all of the arguments above go through identically with this modification.
\end{proof}

\section{Homological stability}\label{s:hs}

In this section, we reduce the proof of Theorem~\ref{thm:disc-minus-Cantor} to Theorem~\ref{thm:hom-stab} (using the results of \S\ref{s:Mather}) and then reduce the proof of Theorem~\ref{thm:hom-stab} to the high-connectivity of certain simplicial complexes, which will be proven in \S\ref{s:high-connectivity}.

First, in \S\ref{s:homogeneous-for-inf-surfaces}, we construct a homogeneous category for big mapping class groups; precisely, the goal of this subsection is to construct a pre-braided homogeneous category whose objects are (words on the alphabet of) decorated surfaces of finite or infinite type: see Proposition \ref{prop:pre-braided-homogeneous}. Using this setup and \cite[Thm.~A]{RWW17}, together with Proposition \ref{prop:acyclicity-assuming-stability}, we next show in \S\ref{s:reduction} that both Theorem~\ref{thm:disc-minus-Cantor} (acyclicity) and Theorem~\ref{thm:hom-stab} (homological stability) reduce to proving that a certain family of semi-simplicial sets $W_n(A,X)_\bullet$ is highly-connected. In \S\ref{s:simplicial-sets-complexes} we show that one may pass to the associated simplicial complexes $S_n(A,X)$ and in \S\ref{s:simplifying} we reduce this further to the statement that a certain (simpler) family of simplicial complexes $TC_n(A,X)$ is highly-connected, which will be proven in \S\ref{s:high-connectivity}.

\subsection{The homogeneous category for infinite-type surfaces}\label{s:homogeneous-for-inf-surfaces}

In this subsection, we build the homogeneous categories needed for the proof of homological stability of infinite-type surfaces. We work essentially with the construction of \cite[\S 5.6]{RWW17}, but allow also non-compact surfaces. In the infinite-type setting we face new technical difficulties: for example, the boundary sum operation no longer satisfies the cancellation property; this will be resolved following \cite[Remark 1.11]{RWW17}.

Recall that, in this paper, surfaces are assumed to be second countable, connected and orientable, and their boundaries are assumed to be compact (thus a disjoint union of finitely many circles). Let $\cM$ be the groupoid whose objects are \emph{decorated surfaces} $(S,I)$, where $S$ is a surface with at least one boundary component and $I \colon [-1,1] \hookrightarrow \partial S$ is a parametrised interval in its boundary. The boundary component of $S$ containing $I$ is called the \emph{based boundary} of $S$ and denoted by $\partial_0 S$. The morphisms of $\cM$ are isotopy classes of homeomorphisms restricting to the identity on $I$. (Note that, by definition, the mapping class group of a surface is the group of isotopy classes of homeomorphisms of $S$ that restrict to the identity on $\partial S$.) The following well-known fact allows us to switch freely between the smooth and topological categories. We include a sketch of a proof of this fact in Appendix \ref{appendix-diff-homeo} to emphasise that it does not depend on assuming that surfaces are of finite type.

\begin{lem}
\label{diff-to-homeo}
For any smooth surface $S$, the forgetful map $\Diff(S,\partial S) \to \Homeo(S,\partial S)$ is a weak homotopy equivalence of topological groups.
\end{lem}

Just as in \cite[\S 5.6]{RWW17}, the boundary connected sum $\natural$ gives a braided monoidal structure on $\cM$. Given two decorated surfaces $(S_1,I_1)$ and $(S_2,I_2)$, their boundary connected sum  $(S_1\natural S_2,I_1\natural I_2)$ is defined to be the surface obtained by gluing $S_1$ and $S_2$ along the right half-interval $I_1^+ \subset \partial S_1$ and the left half-interval $I_2^- \subset \partial S_2$, together with the embedded interval $I_1\natural I_2 \subset \partial(S_1 \natural S_2)$ given by the union $I_1^-\cup I_2^+$ of the other two half-intervals. For an illustration, see \cite[Fig.~2(b), p.~609]{RWW17}. Extending this in the obvious way to morphisms, we obtain an operation $\natural \colon \cM \times \cM \to \cM$ that is associative up to obvious ``bracket-shuffling'' isomorphisms. Although there are isomorphisms $\bD^2 \natural S \cong S \cong S \natural \bD^2$ given by ``absorbing'' the disc into a collar neighbourhood of $S$, these do not satisfy the unit constraints for a monoidal category. Instead, we formally adjoin an additional object $\bI$ to $\cM$ with $\mathrm{End}_\cM(\bI) = \{\mathrm{id}_\bI\}$ and with no morphisms to or from the other objects of $\cM$. We then extend $\natural$ in the unique way so that $\bI$ is a strict unit. Finally, by \cite[Theorem~4.3]{Schauenburg2001}, we may also force \emph{strict} associativity (without changing the underlying category or the unit object) by making careful choices of concrete (set-theoretic) realisations of $S_1 \natural S_2$ for each $S_1$ and $S_2$. Thus $\cM$ is a strict monidal groupoid. It may be equipped with a braiding given by a half Dehn twist in a neighbourhood of the union of based boundaries $\partial_0 S_1 \cup \partial_0 S_2 \subset S_1 \natural S_2$, as pictured in \cite[Fig.~2(c), p.~609]{RWW17}.

We now consider the Quillen bracket category $U\cM$ associated to the braided monoidal groupoid $\cM$ (the construction $U(-)$ is recalled in \S\ref{s:homogeneous-categories}). Since $\cM$ has no zero divisors, $U\cM$ is a pre-braided monoidal category by \cite[Proposition~1.8]{RWW17}. Next, one would like to apply \cite[Theorem~1.10]{RWW17} (recalled as Theorem \ref{thm-grpd-hgca} in \S\ref{s:homogeneous-categories}) to show that $U\cM$ is also a \emph{homogeneous} category. However, this is impossible since condition \eqref{cond-cancellation} (cancellation for the monoid of objects of $\cM$) fails, due to the self-similarity of certain infinite-type surfaces. On the other hand, condition \eqref{cond-injectivity} does hold:

\begin{lem}
\label{lem:injective-map-of-Aut}
Let $(S_1,I_1)$, $(S_2,I_2)$ be two objects in $\cM$. Then the natural map
\begin{equation}
\label{eq:natural-morphism-1}
\iota_\ast \colon \Aut_\cM(S_1,I_1) \longrightarrow \Aut_\cM(S_1 \natural S_2,I_1\natural I_2)
\end{equation}
is injective.
\end{lem}

We give a different proof than that of \cite[Proposition 5.18]{RWW17}.

\begin{proof}
Let $f\in \Aut_\cM(S_1,I_1)$ and suppose $\iota_\ast(f)=1\in \Aut_\cM(S_1 \natural S_2,I_1\natural I_2)$. Our goal is to prove that $f$ is also isotopic to the identity. By the Alexander method (see for example \cite[Proposition 2.8]{FaMa11} for the finite-type case and see \cite{HMV19} for the infinite-type case), it will suffice to show that $f$ preserves the isotopy classes of all simple closed curves and all simple arcs in $S_1$. Let $C$ be a simple closed curve or a simple arc in $S_1$. Since $\iota_*(f) = 1$, we know that $f(C)$ and $C$ are isotopic, thus in particular homotopic, in $S_1 \natural S_2$. There is a retraction $S_1 \natural S_2 \to S_1$, so it follows that $f(C)$ and $C$ are homotopic in $S_1$. It then follows from \cite[Theorem~2.1]{Epstein1966} (if $C$ is a curve) or \cite[Theorem~3.1]{Epstein1966} (if $C$ is an arc) that $f(C)$ and $C$ are ambient isotopic in $S_1$ by an ambient isotopy fixing the boundary of $S_1$. (Epstein requires surfaces to be triangulated; our surfaces are second countable, which implies triangulability by \cite{Rado1925}. For arcs, see also \cite{Feustel1966}.) Thus $f$ is isotopic to the identity by the Alexander method.
\end{proof}

Thus the only property preventing $U\cM$ from being a homogeneous category is the failure of cancellation in $\cM$. We therefore force $\cM$ to satisfy cancellation using the trick explained in \cite[Remark~1.11]{RWW17}: we replace $\cM$ with the groupoid $\bar{\cM}$ whose set of objects is the free monoid on the set of objects of $\cM$, with no morphisms between distinct objects, and with automorphism groups given by
\[
\Aut_{\bar{\cM}}(w) = \Aut_\cM(A_1 \natural A_2 \natural \cdots \natural A_r)
\]
for any word $w = A_1 A_2 \cdots A_r$ on the objects $A_1,A_2,\ldots,A_r$ of $\cM$. The monoidal structure on $\bar{\cM}$ is given by juxtaposition of words, and the unit is the empty word $\varnothing$. In particular, although we may have an isomorphism $A\natural B \cong A$ in $\cM$ for some non-unit object $B$, in $\bar{\cM}$ we instead have $\Hom_{\bar{\cM}}(A\oplus B,A) = \Hom_{\bar{\cM}}(AB,A) = \varnothing$ since the words $AB$ and $A$ are not equal. Having fixed the problem of cancellation, we now have:

\begin{prop}
\label{prop:pre-braided-homogeneous}
Let $\bar{\cM}$ be the braided monoidal groupoid decribed above. Then the Quillen bracket category $U\bar{\cM}$ is pre-braided and homogeneous.
\end{prop}
\begin{proof}
By Theorem \ref{thm-grpd-hgca} (Proposition 1.8 and Theorem 1.10 of \cite{RWW17}), it suffices to check that $\bar{\cM}$ has no zero divisors and satisfies cancellation, and that the natural morphism
\begin{equation}
\label{eq:natural-morphism-2}
\Aut_{\bar{\cM}}(w_1) \longrightarrow \Aut_{\bar{\cM}}(w_1 w_2)
\end{equation}
is injective for any two words $w_1,w_2$ in the objects of $\cM$. No zero divisors and cancellation are automatic from the fact that (by definition) the monoid of objects of $\bar{\cM}$ is free. If we write $w_1 = A_1 A_2 \cdots A_r$ and $w_2 = B_1 B_2 \cdots B_s$, the morphism \eqref{eq:natural-morphism-2} may be written as
\[
\Aut_\cM(A_1 \natural A_2 \natural \cdots \natural A_r) \longrightarrow \Aut_\cM(A_1 \natural A_2 \natural \cdots \natural A_r \natural B_1 \natural B_2 \natural \cdots \natural B_s).
\]
This is a composition of $s$ morphisms of the form \eqref{eq:natural-morphism-1}, so it is injective by Lemma \ref{lem:injective-map-of-Aut}.
\end{proof}

We note that we have not changed the isomorphisms of the category by passing from $\bar{\cM}$ to its Quillen bracket construction $U\bar{\cM}$.

\begin{lem}
\label{lem:underlying-groupoid}
There is a canonical isomorphism of groupoids
\[
\bar{\cM} \cong \mathrm{Iso}(U\bar{\cM}),
\]
where $\mathrm{Iso}(\cC)$ denotes the underlying groupoid of $\cC$. In particular, automorphism groups in $U\bar{\cM}$ are the same as automorphism groups in $\bar{\cM}$.
\end{lem}
\begin{proof}
By \cite[Proposition 1.7]{RWW17}, it suffices to check that the automorphism group $\Aut_{\bar{\cM}}(\varnothing)$ is trivial and that the monoid of (isomorphism classes of) objects of $\bar{\cM}$ has no zero-divisors. The first statement follows from the fact that $\Aut_{\bar{\cM}}(\varnothing) = \Aut_\cM(\bI) = \{\mathrm{id}_\bI\}$ by definition of $\bar{\cM}$ and by definition of the (formally adjoined) unit of $\cM$. The second statement follows from the fact that the monoid of (isomorphism classes of) objects of $\bar{\cM}$ is free.
\end{proof}

\subsection{Reduction to high-connectivity}\label{s:reduction}

Applied to the pre-braided homogeneous category $U\bar{\cM}$ (see Proposition \ref{prop:pre-braided-homogeneous}), Theorem A of \cite{RWW17}, applied to the constant coefficient system $\bZ$, says the following. Let $A$ and $X$ be two decorated surfaces (possibly of infinite type), each with exactly one boundary component.

\begin{thm}[{\cite[Theorem~A~applied to $U\bar{\cM}$]{RWW17}}]
\label{thm:RWW}
Suppose that the semi-simplicial set $W_n(A,X)_\bullet$ from Definition \ref{defn-W_n(A,X)} is $(\tfrac{n-2}{2})$-connected for each $n\geq 1$. Then the stabilisation map
\begin{equation}
\label{eq:stab-map-general}
\mathrm{Map}(A \natural X^{\natural n}) \longrightarrow \mathrm{Map}(A \natural X^{\natural n+1})
\end{equation}
induces isomorphisms \textup{(}surjections\textup{)} on $H_i(-;\bZ)$ for each $n\geq 1$ and for $i\leq \tfrac{n-2}{2}$ \textup{(}for $i\leq \tfrac{n}{2}$\textup{)}.
\end{thm}

\begin{rem}
We note that the statement in \cite{RWW17} has a stabilisation map of the form $\mathrm{Aut}_{\cU\bar{\cM}}(A \natural X^{\natural n}) \to \mathrm{Aut}_{\cU\bar{\cM}}(A \natural X^{\natural n+1})$ in place of \eqref{eq:stab-map-general}; we have applied the identifications
\begin{align*}
\mathrm{Aut}_{\cU\bar{\cM}}(A \natural X^{\natural n}) &= \mathrm{Aut}_{\bar{\cM}}(A \natural X^{\natural n}) \\
&= \mathrm{Aut}_{\cM}(A \natural X^{\natural n}) \\
&= \mathrm{Map}(A \natural X^{\natural n}),
\end{align*}
which follow by Lemma \ref{lem:underlying-groupoid} and the definitions of $\bar{\cM}$ and $\cM$ respectively. Notice that \eqref{eq:stab-map-general} is the map \eqref{eq:stab-maps} from the statement of Theorem \ref{thm:hom-stab}.
\end{rem}

\begin{cor}
\label{coro:reduction-to-high-conn}
Let $\Sigma$ be a connected surface without boundary. Suppose that the semi-simplicial set $W_n(\bD^2,\mathfrak{B}(\Sigma))_\bullet$ from Definition \ref{defn-W_n(A,X)} is $(\tfrac{n-2}{2})$-connected for each $n\geq 1$. Then the mapping class group $\Map(\mathfrak{B}(\Sigma))$ is acyclic.
\end{cor}
\begin{proof}
Consider the homomorphism
\[
e_n \colon \Map(\mathfrak{B}(\Sigma)^{\natural n}) \longrightarrow \Map(\mathfrak{B}(\Sigma)^{\natural n+1})
\]
induced by including $\mathfrak{B}(\Sigma)^{\natural n}$ into $\mathfrak{B}(\Sigma)^{\natural n+1}$ and extending diffeomorphisms by the identity on the new $\mathfrak{B}(\Sigma)$ summand. By Theorem \ref{thm:RWW}, this map induces isomorphisms on $H_i(-;\bZ)$ for all $n\geq 2i+2$. Via the self-similarity of the binary tree $\mathfrak{B}$, there is a homeomorphism $\mathfrak{B}(\Sigma) \natural \mathfrak{B}(\Sigma) \cong \mathfrak{B}(\Sigma)$, which, by iteration, induces a homeomorphism $\mathfrak{B}(\Sigma)^{\natural n} \cong \mathfrak{B}(\Sigma)$ for all $n\geq 1$. Using this, we may identify the homomorphism $e_n$ with the homomorphism $e_1 = e = \eqref{eq:stabilisation-e}$ for all $n\geq 1$ and deduce that this map is a homology isomorphism in all degrees. Proposition \ref{prop:acyclicity-assuming-stability} then implies that $\Map(\mathfrak{B}(\Sigma))$ is acyclic.
\end{proof}

Proving Theorems \ref{thm:disc-minus-Cantor} and \ref{thm:hom-stab} has therefore been reduced to proving that the semi-simplicial set $W_n(A,X)_\bullet$ is $(\tfrac{n-2}{2})$-connected for any $A$ and for each $X$ in the list in the statement of Theorem \ref{thm:hom-stab}. (Theorem \ref{thm:disc-minus-Cantor} just requires the case of $A = \bD^2$ and $X = \mathfrak{B}(\Sigma)$.)

\subsection{From semi-simplicial sets to simplicial complexes}
\label{s:simplicial-sets-complexes}

As a first step towards studying the connectivity of $W_n(A,X)_\bullet$, we will replace it by its associated simplicial complex $S_n(A,X)$ (Proposition \ref{prop:from-W-to-S}).

\begin{defn}
For any semi-simplicial set $Y_\bullet$, the simplicial complex $K(Y_\bullet)$ is defined to have objects $Y_0$ and a collection $\{y_0,\ldots,y_p\} \subseteq Y_0$ spans a $p$-simplex if and only if there exists $\sigma \in Y_p$ such that
\[
\{ f_0(\sigma) , \ldots , f_p(\sigma) \} = \{ y_0 , \ldots , y_p \},
\]
where $f_0,\ldots,f_p \colon Y_p \to Y_0$ are the $p+1$ distinct compositions of face maps of $Y_\bullet$.
\end{defn}

The following lemma is immediate from the definitions of geometric realisations of semi-simplicial sets and simplicial complexes respectively.

\begin{lem}
\label{lem:homeomorphism-geometric-realisations}
Suppose that the semi-simplicial set $Y_\bullet$ has the property that, for every simplex $\sigma \in Y_p$, the vertices $f_0(\sigma),\ldots,f_p(\sigma) \in Y_0$ are distinct. Moreover, suppose that each $\sigma \in Y_p$ is determined by the unordered set $\{ f_0(\sigma),\ldots,f_p(\sigma) \} \subseteq Y_0$. Then there is a homeomorphism
\[
\lvert Y_\bullet \rvert \cong \lvert K(Y_\bullet) \rvert .
\]
\end{lem}

In our setting, we will denote the associated simplicial complex as follows (this agrees with Definition \ref{defn:Sn(X,A)} in \S\ref{s:homogeneous-categories}, which in turn is \cite[Definition~2.8]{RWW17}).

\begin{defn}
For any pair of decorated surfaces $(A,X)$, let $S_n(A,X) = K(W_n(A,X)_\bullet)$.
\end{defn}

\begin{prop}
\label{prop:from-W-to-S}
For a pair of decorated surfaces $(A,X)$ with $X \neq \bD^2$, the associated semi-simplicial set $W_n(A,X)_\bullet$ satisfies the conditions of Lemma \ref{lem:homeomorphism-geometric-realisations}. Thus we have homeomorphisms
\[
\lvert W_n(A,X)_\bullet \rvert \cong \lvert S_n(A,X) \rvert .
\]
\end{prop}
\begin{proof}
First of all, we recall from \cite[Definition~2.5]{RWW17} that a homogeneous category $\cC$ is \emph{locally standard} at a pair of objects $(A,X)$ if the two canonical morphisms $X \to A \oplus X \oplus X$ are distinct and the map
\begin{equation}
\label{eq:locally-standard-map}
\Hom_\cC(X , A \oplus X^{\oplus n-1}) \longrightarrow \Hom_\cC(X , A \oplus X^{\oplus n})
\end{equation}
sending $f$ to $f \oplus \iota_X$ is injective. By \cite[Proposition~2.6]{RWW17}, this condition implies that any simplex $\sigma \in W_n(A,X)_p$ has distinct vertices $f_0(\sigma),\ldots,f_p(\sigma)$ in $W_n(A,X)_0$ and also that each $\sigma \in W_n(A,X)_p$ is determined by its \emph{ordered} list of vertices $(f_0(\sigma),\ldots,f_p(\sigma)) \in W_n(A,X)_0^{p+1}$. Since $U\bar{\cM}$ is homogeneous (Proposition~\ref{prop:pre-braided-homogeneous}), it will thus suffice to show that:
\begin{enumerate}
\item\label{prop-injective} the map \eqref{eq:locally-standard-map} is injective;
\item\label{prop-distinct} the two canonical morphisms $X \to A \oplus X \oplus X$ are distinct;
\item\label{prop-permutation} for $(v_0,\ldots,v_p) \in W_n(A,X)_0^{p+1}$, there is at most one permutation $\tau$ of $\{0,1,\ldots,p\}$ such that $(v_{\tau(0)},\ldots,v_{\tau(p)})$ is the ordered list of the vertices of a $p$-simplex in $W_n(A,X)_p$.
\end{enumerate}
We will begin by proving property \eqref{prop-injective}. Unwinding the definitions, we first observe that the set $\Hom_{U\bar{\cM}}(X,A \oplus X^{\oplus n})$ is in natural bijection with the quotient set $\Map(A \natural X^{\natural n}) / \Map(A \natural X^{\natural (n-1)})$, where we write $\Map(-) = \pi_0(\Homeo_I(-)) = \Aut_{\cM}(-)$ and where the map \eqref{eq:locally-standard-map} is induced by taking quotients vertically in the square of group homomorphisms
\begin{equation}
\label{eq:square-of-inclusions}
\begin{tikzcd}
\Map(A \natural X^{\natural (n-2)}) \ar[rr,hook] \ar[d,hook] && \Map(A \natural X^{\natural (n-1)}) \ar[d,hook] \\
\Map(A \natural X^{\natural (n-1)}) \ar[r,hook] & \Map(A \natural X^{\natural n}) \ar[r,"b"] & \Map(A \natural X^{\natural n}),
\end{tikzcd}
\end{equation}
where the map marked ``$b$'' is the braiding on the last two copies of $X$ in $A \natural X^{\natural n}$ and the other maps (induced by extending homeomorphisms by the identity) are injective by Lemma \ref{lem:injective-map-of-Aut}. The map \eqref{eq:locally-standard-map} is injective if and only if the square \eqref{eq:square-of-inclusions} is a pullback square, which is equivalent to the following claim, where we write $A \natural X^{\natural n} = \bar{A} \natural X_1 \natural X_2$ with $\bar{A} = A \natural X^{\natural (n-2)}$ and $X_1 = X_2 = X$.
\begin{itemizeb}
\item[\textit{Claim.}] Suppose that an element $g \in \Map(A \natural X^{\natural n})$ admits representatives $g = [\varphi_1] = [\varphi_2]$, where $\varphi_1$ restricts to the identity on $X_1$ and $\varphi_2$ restricts to the identity on $X_2$. Then $g$ admits a representative $[\varphi_{12}]$, where $\varphi_{12}$ restricts to the identity on $X_1 \natural X_2$.
\end{itemizeb}
To prove this claim, first notice that the subsurface $X_1$ of $\bar{A} \natural X_1 \natural X_2$ disconnects the surface into the two pieces $\bar{A}$ and $X_2$. Because of this, since $\varphi_1$ restricts to the identity on $X_1$, it must be equal to $\varphi'_1 \natural \mathrm{id}_{X_1} \natural \psi$ for some homeomorphisms $\varphi'_1$ and $\psi$ of $\bar{A}$ and $X_2$ respectively. It will suffice to show that the homeomorphism $\varphi''_1 = \mathrm{id}_{X_1} \natural \psi$ of $X_1 \natural X_2$ is isotopic to the identity (fixing the boundary of $X_1 \natural X_2$). Now let $\cA$ be a stable Alexander system of essential arcs and curves on $X_1 \natural X_2$ having the property that each arc or curve in $\cA$ is contained either in $X_1$ or in $X_2$. (For infinite-type surfaces this may be ensured by following the construction of stable Alexander systems in \cite{HMV19}: they are built from the boundaries of a stable exhaustion, a ``dual'' curve for each boundary of the exhaustion and a finite stable Alexander system for each stratum of the exhaustion. As long as we ensure that the stable exhaustion respects the decomposition of $X_1 \natural X_2$ into $X_1$ and $X_2$, we may arrange that the stable Alexander system does too.) Let $C$ be an arc or curve in $\cA$. If $C$ lies in $X_1$, then $\varphi''_1(C) = C$. If $C$ lies in $X_2$, then $\varphi''_1(C) = \varphi_1(C)$ and $\varphi_2(C) = C$ are arcs or curves in $X_2$ that are isotopic in $A \natural X^{\natural n}$. Via a retraction $A \natural X^{\natural n} \to X_2$, they are homotopic in $X_2$ and thus isotopic in $X_2$ by \cite[Theorem~2.1~or~3.1]{Epstein1966}. Hence, by the Alexander trick, the homeomorphism $\varphi''_1$ of $X_1 \natural X_2$ is isotopic to the identity. This completes the proof of the Claim and therefore also of property \eqref{prop-injective}.

We next prove property \eqref{prop-distinct}. As noted above, morphisms $X \to A \oplus X \oplus X$ in $U\bar{\cM}$ correspond to elements of $\Map(A \natural X_1 \natural X_2) / \Map(A \natural X_1)$, where we write $X_1 = X_2 = X$ as a notational device to distinguish the two copies of $X$ in the boundary connected sum. The two canonical morphisms correspond, under this identification, to the cosets $[\mathrm{id}]$ and $[\mathrm{id}_A \natural b]$, where $b$ denotes the braiding of $X_1 \natural X_2$. Property \eqref{prop-distinct} is then equivalent to the claim that $\mathrm{id}_A \natural b$ is not isotopic to a homeomorphism of the form $\varphi \natural \mathrm{id}_{X_2}$. To see this, consider the induced actions on $H_1(A \natural X_1 \natural X_2)$. Any homeomorphism isotopic to one of the form $\varphi \natural \mathrm{id}_{X_2}$ must act trivially on all homology classes supported on $X_2$. On the other hand, since $X_2 = X$ is not the disc, it supports non-trivial homology classes, which are acted on non-trivially by the braiding $\mathrm{id}_A \natural b$. This establishes property \eqref{prop-distinct}.

Property \eqref{prop-permutation} is most easily seen using the alternative description of the simplices of $W_n(A,X)_\bullet$ that we explain in the next subsection, as ambient isotopy classes of collections of pairwise disjointly embedded copies of $X$ in the interior of $A \natural X^{\natural n}$ equipped with arcs to the basepoint of $A \natural X^{\natural n}$, where the arcs are also pairwise disjoint except at the basepoint. In this description, the ordering of the list of vertices of a simplex corresponds to the intrinsic ordering of the arcs in a small neighbourhood of the basepoint. Changing the ordering of the vertices without changing the unordered set of vertices would therefore entail isotoping these arcs (fixing their endpoints) so that at the end of the isotopy, we have: (i) they are again pairwise disjoint and disjoint from the embedded copies of $X$; (ii) their ordering near the basepoint is different. This is impossible whenever $X$ has non-trivial fundamental group, which is the case since we have assumed that $X$ is not the disc.
\end{proof}

\subsection{Simplifying the simplicial complexes}\label{s:simplifying}

We now discuss in more detail the morphisms $X^{\oplus i} \to A \oplus X^{\oplus n}$ in $U\bar{\cM}$ in order to simplify the simplicial complex $S_n(A,X)$. We first give an explicit topological description of this set of morphisms (Proposition \ref{prop:hom-set-admissible-embeddings}) in terms of \emph{admissible embeddings} (Definition \ref{defn-admissible-embeddings}). This will allow us to replace $S_n(A,X)$ with another simplicial complex where we remember only boundaries of admissible embeddings.

\begin{construction}
Let $Y$ and $Z$ be two decorated surfaces. Let us consider a smooth embedding
\begin{equation}
\label{eq:admissible-embedding}
f \colon (Y,[0,1]) \lhook\joinrel\longrightarrow (Z,[0,1]),
\end{equation}
where the notation $(-,[0,1])$ means that
\begin{itemizeb}
\item $f$ acts by the identity on $[0,1] \subset \partial Y$;
\item $f$ sends $[-1,0) \subset Y$ into the interior of $Z$.
\end{itemizeb}
Let us consider the subsurface
\begin{equation}
\label{eq:complement}
f^\perp := \overline{Z \smallsetminus f(Y)},
\end{equation}
the closure in $Z$ of the complement of $f(Y)$. This has the structure of a decorated surface, with decoration $I \colon [-1,1] \to \partial (f^\perp)$ given by the decoration of $Z$ on $[-1,0]$ and by the composition $f \circ (-1) \colon [0,1] \to [-1,0] \to Y \to Z$ on $[0,1]$, where the middle arrow is the negative part of the decoration of $Y$.
\end{construction}

\begin{defn}
\label{defn-admissible-embeddings}
An embedding $f$ of the form \eqref{eq:admissible-embedding} is called \emph{pre-admissible} if:
\begin{itemizeb}
\item it is a proper embedding, i.e., its image is a closed subset of $Z$,
\item the intersection $f(Y) \cap f^\perp$ is equal to $f([-1,0])$.
\end{itemizeb}
These conditions imply that $Z = f^\perp \natural f(Y)$. When $Y = X^{\natural i}$ and $Z = A \natural X^{\natural n}$, a pre-admissible embedding $f$ is called \emph{admissible} if, in addition:
\begin{itemizeb}
\item the subsurface $f^\perp$ is diffeomorphic -- as a decorated surface -- to $A\natural X^{\natural (n-i)}$.
\end{itemizeb}
The space of admissible embeddings of $X^{\natural i}$ into $A \natural X^{\natural n}$ is denoted by
\[
\Emba_{[0,1]}(X^{\natural i} , A \natural X^{\natural n} ) \subset \Emb(X^{\natural i} , A \natural X^{\natural n})
\]
and is equipped with the topology induced by the smooth Whitney topology on $\Emb(X^{\natural i} , A \natural X^{\natural n})$. Unwinding this definition, one sees that an embedding $X^{\natural i} \hookrightarrow A \natural X^{\natural n}$ is admissible if and only if it extends (along the canonical inclusion) to a diffeomorphism of $A \natural X^{\natural n}$; this is the characterisation that will be needed in the proof of Proposition \ref{prop:hom-set-admissible-embeddings} below.
\end{defn}

We will also need the following elementary lemma, whose proof is immediate.

\begin{lem}
\label{lem:path-lifting-GH}
Let $H \subset G$ be an inclusion of topological groups and suppose that the quotient map $G \twoheadrightarrow G/H$ admits path lifting. Then there is a natural bijection
\[
\pi_0(G) / \pi_0(H) \cong \pi_0(G/H) .
\]
\end{lem}

\begin{prop}
\label{prop:hom-set-admissible-embeddings}
For decorated surfaces $A \neq X$ there is a natural bijection
\[
\Hom_{\cU\bar{\cM}} (X^{\oplus i},A\oplus X^{\oplus n}) \cong \begin{cases}
\pi_0 \bigl( \Emba_{[0,1]}(X^{\natural i} , A \natural X^{\natural n} ) \bigr) & \text{for } i\leq n \\
\varnothing & \text{for } i>n.
\end{cases}
\]
\end{prop}
\begin{proof}
By definition, $\Hom_{\cU\bar{\cM}} (X^{\oplus i},A\oplus X^{\oplus n})$ is the set of pairs $(Y,f)$ where $f \colon Y\oplus X^{\oplus i} \to A\oplus X^{\oplus n}$ is a morphism in $\Hom_{\bar{\cM}}(Y\oplus X^{\oplus i}, A\oplus X^{\oplus n})$ up to a certain equivalence relation. By definition, $\bar{\cM}$ has no morphisms between distinct objects, so this forces $Y \oplus X^{\oplus i} = A \oplus X^{\oplus n}$. Moreover, objects of $\bar{\cM}$ are words of decorated surfaces and $\oplus$ is concatenation of words, so by cancellation this implies:
\begin{itemizeb}
\item that $Y = A \oplus X^{\oplus (n-i)}$, if $i\leq n$;
\item a contradiction, if $i>n$.
\end{itemizeb}
Hence there are no morphisms if $i>n$ and we may assume from now on that $i\leq n$. Thus we have no choice about $Y$ and $f$ is an automorphism of $A \oplus X^{\oplus n}$ in $\bar{\cM}$, which is an automorphism of $A \natural X^{\natural n}$ in $\cM$, which is an isotopy class of self-homeomorphisms of $A \natural X^{\natural n}$ fixing $I$ pointwise. Two such pairs $(A\oplus X^{\oplus (n-i)},f_1)$ and $(A\oplus X^{\oplus (n-i)},f_2)$ are equivalent if and only if there exists a self-homeomorphism $g$ of $A\natural X^{\natural (n-i)}$ fixing $I$ pointwise making the following diagram commute:
 \[
\begin{tikzcd}
A\natural X^{\natural(n-i)} \natural X \ar[rr,"f"] \ar[d,"{g\natural \id_X}",swap]  && A\natural X^{\natural n}    \\
 A\natural X^{\natural (n-i)}\natural X \ar[urr,"{f'}",swap]  &&
\end{tikzcd}
\]
Thus we have a bijection
\begin{equation}
\label{eq:bijection-1}
\Hom_{\cU\bar{\cM}} (X^{\oplus i},A\oplus X^{\oplus n}) \longleftrightarrow \pi_0(\Homeo_I(\bar{A} \natural X^{\natural i})) / \pi_0(\Homeo_I(\bar{A})),
\end{equation}
where we abbreviate $\bar{A} = A \natural X^{\natural n-i}$. From the commutative diagram
\[
\begin{tikzcd}
\pi_0(\Diff_I(\bar{A})) \ar[d] \ar[r] & \pi_0(\Diff_I(\bar{A} \natural X^{\natural i})) \ar[d] \\
\pi_0(\Homeo_I(\bar{A})) \ar[r] & \pi_0(\Homeo_I(\bar{A} \natural X^{\natural i}))
\end{tikzcd}
\]
and Lemma \ref{diff-to-homeo}, which tells us that the vertical maps are bijections, we have a natural bijection
\begin{equation}
\label{eq:bijection-2}
\pi_0(\Homeo_I(\bar{A} \natural X^{\natural i})) / \pi_0(\Homeo_I(\bar{A})) \longleftrightarrow \pi_0(\Diff_I(\bar{A} \natural X^{\natural i})) / \pi_0(\Diff_I(\bar{A})) .
\end{equation}
Next, we use the fact that the restriction map
\begin{equation}
\label{eq:restriction-map}
\Diff_I(\bar{A} \natural X^{\natural i}) \longrightarrow \Emb_{[0,1]}(X^{\natural i},\bar{A} \natural X^{\natural i}),
\end{equation}
given by pre-composition with the inclusion $X^{\natural i} \hookrightarrow \bar{A} \natural X^{\natural i}$, is a fibre bundle. (This is the Isotopy Extension Theorem.) It therefore remains a fibre bundle when we replace its codomain with its image, which is, by definition, the subspace $\Emba_{[0,1]}(X^{\natural i},\bar{A} \natural X^{\natural i})$ of \emph{admissible} embeddings. The (non-empty) fibres of \eqref{eq:restriction-map} are precisely the orbits of the action of $\Diff_I(\bar{A})$, so the restriction map factors as
\begin{equation}
\label{eq:restriction-map-triangle}
\begin{tikzcd}
& \Diff_I(\bar{A} \natural X^{\natural i}) \ar[dl,two heads] \ar[dr,two heads] & \\
\Diff_I(\bar{A} \natural X^{\natural i}) / \Diff_I(\bar{A}) \ar[rr] && \Emba_{[0,1]}(X^{\natural i} , A \natural X^{\natural n} )
\end{tikzcd}
\end{equation}
where the bottom horizontal map is a continuous bijection. The left diagonal map in \eqref{eq:restriction-map-triangle} is a quotient map (by definition) and the right diagonal map in \eqref{eq:restriction-map-triangle} is also a quotient map (because it is a fibre bundle); thus the bottom horizontal map is in fact a homeomorphism. In particular, it induces a bijection
\begin{equation}
\label{eq:bijection-3}
\pi_0 \bigl( \Diff_I(\bar{A} \natural X^{\natural i}) / \Diff_I(\bar{A}) \bigr) \longleftrightarrow \pi_0 \bigl( \Emba_{[0,1]}(X^{\natural i} , A \natural X^{\natural n} ) \bigr) .
\end{equation}
Finally, since the right diagonal map in \eqref{eq:restriction-map-triangle} is a fibre bundle and the bottom horizontal map in \eqref{eq:restriction-map-triangle} is a homeomorphism, the left diagonal map in \eqref{eq:restriction-map-triangle} is also a fibre bundle; in particular it admits path lifting. Thus Lemma \ref{lem:path-lifting-GH} implies that there is a natural bijection
\begin{equation}
\label{eq:bijection-4}
\pi_0(\Diff_I(\bar{A} \natural X^{\natural i})) / \pi_0(\Diff_I(\bar{A})) \longleftrightarrow \pi_0 \bigl( \Diff_I(\bar{A} \natural X^{\natural i}) / \Diff_I(\bar{A}) \bigr) .
\end{equation}
The claim now follows by composing the bijections \eqref{eq:bijection-1}, \eqref{eq:bijection-2}, \eqref{eq:bijection-4} and \eqref{eq:bijection-3}.
\end{proof}

\begin{rem}
This identification of morphisms in a Quillen bracket category of manifolds as embeddings satisfying a certain admissibility condition is similar to \cite[Proposition~4.8]{PalmerSoulie2019}, which considers topologically-enriched Quillen bracket categories of manifolds.
\end{rem}

By Proposition \ref{prop:hom-set-admissible-embeddings} and the definition of $W_n(A,X)_\bullet$, we deduce:

\begin{cor}
\label{cor:hom-set-admissible-embeddings}
For each $p\geq 0$, there is a natural bijection
\[
W_n(A,X)_p \cong \begin{cases}
\pi_0 \bigl( \Emba_{[0,1]}(X^{\natural (p+1)} , A \natural X^{\natural n} ) \bigr) & \text{for } p\leq n-1 \\
\varnothing & \text{for } p\geq n.
\end{cases}
\]
Moreover, under these identifications, the face maps of $W_n(A,X)_\bullet$ correspond to pre-composition with embeddings $X^{\natural (i-1)} \hookrightarrow X^{\natural i}$ given by the standard inclusion composed with appropriate braidings of $X^{\natural i}$.
\end{cor}

This also gives a topological description of the vertices and simplices of the associated simplicial complex $S_n(A,X)$ in terms of admissible embeddings. We now use this description to observe that $S_n(A,X)$ is a complete join over the simplicial complex where we forget most information about the admissible embeddings and only remember their restrictions to the based boundary. This will allow us to progressively simplify $S_n(A,X)$ to a more manageable simplicial complex $TC_n(A,X)$ of ``tethered $X$-curves'', whose high-connectivity will imply high-connectivity of $S_n(A,X)$; see Theorem \ref{thm:reduction-to-TC}. We do this simplification in the remainder of this section, and then in \S\ref{s:high-connectivity} we prove high-connectivity of $TC_n(A,X)$ in each of the three settings listed in the statement of Theorem \ref{thm:hom-stab}.

\begin{defn}
A based, embedded loop
\[
\alpha \colon (S^1,\{*\}) \longrightarrow (A\natural X^{\natural n},\{0\})
\]
is called \emph{$X$-admissible} if there exists an admissible embedding $f \colon (X,[0,1]) \hookrightarrow (A\natural X^{\natural n},[0,1])$ (see Definition \ref{defn-admissible-embeddings}) such that $f|_{\partial_0 X} = \alpha$. Similarly, a loop $\alpha$ is called \emph{$X^{\natural i}$-admissible} if there exists an admissible embedding $f \colon (X^{\natural i},[0,1]) \hookrightarrow (A\natural X^{\natural n},[0,1])$ such that $f|_{\partial_0 X^{\natural i}} = \alpha$.
\end{defn}

\begin{defn}
\label{defn:complex-Un(A,X)}
Let $U_n(A,X)$ be the following simplicial complex:
\begin{enumerate}
\item a vertex of $U_n(A,X)$ is an isotopy class of $X$-admissible loops;
\item a $(p+1)$-tuple $[\alpha_0], \ldots, [\alpha_p]$ of vertices spans a $p$-simplex of $U_n(A,X)$ if there exist representatives $\alpha_0,\ldots,\alpha_p$ of the isotopy classes that are pairwise disjoint except at the basepoint $0 \in A\natural X^{\natural n}$ and moreover their concatenation $\alpha_0 \cdots \alpha_p$ is $X^{\natural (p+1)}$-admissible up to isotopy.
\end{enumerate}
\end{defn}

By Corollary \ref{cor:hom-set-admissible-embeddings}, the vertices of $S_n(A,X)$, which are by definition also the vertices of $W_n(A,X)_\bullet$, are isotopy classes of admissible embeddings $f \colon (X,[0,1]) \hookrightarrow (A\natural X^{\natural n},[0,1])$. There is therefore a map of simplicial complexes
\begin{equation}
\label{eq:forgetful-map-from-S-to-U}
\cF \colon S_n(A,X) \longrightarrow U_n(A,X)
\end{equation}
given by sending $[f]$ to $[f|_{\partial_0 X}]$. Let us verify that it does indeed send simplices to simplices. A simplex of $S_n(A,X)$ is a collection of admissible embeddings of $X$ that arises from an admissible embedding of $X^{\natural (p+1)}$ by restricting to the last copy of $X$ in the boundary connected sum $X^{\natural (p+1)}$ and then applying the braiding of $\cM$ ($i$ times, for each $0\leq i\leq p$). The based boundaries of these admissible embeddings are not disjoint on the nose, but they may easily be made pairwise disjoint (except at the basepoint) by shrinking them slightly in a collar neighbourhood of the based boundary of $X$. We record the following key observation.

\begin{obs}
\label{obs-cjoin}
When $X \neq \bD^2$, the forgetful map \eqref{eq:forgetful-map-from-S-to-U} is a complete join.
\end{obs}
\begin{proof}
The map \eqref{eq:forgetful-map-from-S-to-U} is surjective by definition of the vertices of $U_n(A,X)$, so we must check conditions \eqref{complete-join-injective} and \eqref{complete-join-lifting} of Definition \ref{defn:complete-join}. Failure of condition \eqref{complete-join-injective} (injectivity on each simplex) would imply, writing $\bar{A} = A \natural X^{\natural (n-2)}$ and $X_1 = X_2 = X$, that the based boundaries $\partial_0 X_1$ and $\partial_0 X_2$ are isotopic in $\bar{A} \natural X_1 \natural X_2$. But since $X$ is not a disc, the homology classes $[\partial_0 X_1], [\partial_0 X_2] \in H_1(\bar{A}\natural X_1\natural X_2)$ must be distinct. Thus condition \eqref{complete-join-injective} holds. Condition \eqref{complete-join-lifting} is the statement that a simplex of $U_n(A,X)$ may be lifted to a simplex of $S_n(A,X)$ by simply choosing lifts of each of its vertices, without requiring any compatibility between the lifts of vertices. This holds in our case because a collection of pairwise disjoint $X$-admissible loops representing a $p$-simplex in $U_n(A,X)$ cuts out $p+1$ pieces from the surface $A \natural X^{\natural n}$, which are each homeomorphic to $X$, and lifting to a $p$-simplex of $S_n(A,X)$ consists in \emph{choosing} a homeomorphism with $X$ for each of these pieces; this may be done independently for each piece. Thus condition \eqref{complete-join-lifting} also holds, so \eqref{eq:forgetful-map-from-S-to-U} is a complete join.
\end{proof}

By Proposition \ref{prop-join-conn} it will suffice, for high-connectivity of $S_n(A,X)$, to show that the complex $U_n(A,X)$ is weakly Cohen-Macaulay (\emph{wCM}). We now give another, isomorphic description of the simplicial complex $U_n(A,X)$ in terms of \emph{tethered curves} (cf.~\cite{HatcherVogtmann2017}), which are embeddings of the ``lasso'' $L = [0,2]/(1 \sim 2)$ into the surface $A\natural X^{\natural n}$.

\begin{defn}
A tethered curve $\alpha \colon L \hookrightarrow A \natural X^{\natural n}$ is called a \emph{tethered $X$-curve} if:
\begin{enumerate}
    \item the arc of $\alpha$ starts from the basepoint, i.e.~$L(0) = 0 \in I \subset A \natural X^{\natural n}$;
    \item the curve $\alpha([1,2])$ bounds a subsurface of $A \natural X^{\natural n}$ that is homeomorphic to $X$;
    \item the surface obtained from $A \natural X^{\natural n}$ by cutting along the arc $\alpha([0,1])$ and removing the interior of the subsurface bounded by $\alpha([1,2])$ is homeomorphic to $A\natural X^{\natural (n-1)}$.
\end{enumerate} 
\end{defn}

\begin{defn}
\label{defn-teth-X-cur}
For a decorated surface $A\natural X^{\natural n}$, its \emph{tethered $X$-curve complex} $TC_n(A,X)$ is defined to be the following simplicial complex:
\begin{enumerate}
\item a vertex of $TC_n(A,X)$ is an isotopy class of tethered $X$-curves;
\item a $(p+1)$-tuple $[\alpha_0], \ldots, [\alpha_p]$ of vertices spans a $p$-simplex of $TC_n(A,X)$ if there exist representatives $\alpha_0,\ldots,\alpha_p$ of the isotopy classes that are pairwise disjoint except at the basepoint $0 \in A\natural X^{\natural n}$ and moreover the complementary surface obtained by cutting out the $p+1$ subsurfaces bounded by $\alpha_0,\ldots,\alpha_p$ is homeomorphic to $A\natural X^{\natural (n-p-1)}$.
\end{enumerate}
\end{defn}

Unwinding the definitions, we see that isotopy classes of $X$-admissible loops in $A\natural X^{\natural n}$ are in one-to-one correspondence with isotopy classes of tethered $X$-curves in $A\natural X^{\natural n}$ and this correspondence respects the property of being representable disjointly with complementary subsurface homeomorphic to $A\natural X^{\natural (n-p-1)}$. We therefore have:

\begin{lem}
\label{lem-TC-U-cplx-iso}
The simplicial complexes $TC_n(A, X)$ and $U_n(A,X)$ are isomorphic.
\end{lem}

To finish this section, we show that high-connectivity of $TC_n(A,X)$ implies high-connectivity of $W_n(A,X)_\bullet$.

\begin{thm}
\label{thm:reduction-to-TC}
Let $A$ and $X$ be two decorated surfaces with $X \neq \bD^2$ and fix an integer $k\geq 1$. If $TC_n(A,X)$ is $(\frac{n-2}{k})$-connected for all $n\geq 1$, then $W_n(A,X)_\bullet$ is $(\frac{n-2}{k})$-connected for all $n\geq 1$.
\end{thm}
\begin{proof}
First notice that the link of any $p$-simplex in $TC_n(A, X)$ is isomorphic to $TC_{n-p-1}(A, X)$, whose connectivity is at least $\frac{n-p-3}{k} \geq (\frac{n-2}{k})-p-1$ since $k\geq 1$. Hence $TC_n(A, X)$ is \emph{wCM} of dimension $(\frac{n-2}{k})+1$. By Lemma \ref{lem-TC-U-cplx-iso}, this means that $U_n(A,X)$ is \emph{wCM} of dimension $(\frac{n-2}{k})+1$. Observation \ref{obs-cjoin} and \cite[Proposition~3.5]{HW10} (recalled as Proposition \ref{prop-join-conn}) then imply that $S_n(A,X)$ is \emph{wCM} of dimension $(\frac{n-2}{k})+1$, in particular it is $(\frac{n-2}{k})$-connected. Proposition \ref{prop:from-W-to-S} then implies that $W_n(A,X)_\bullet$ is also $(\frac{n-2}{k})$-connected.
\end{proof}

\section{Proof of high-connectivity}
\label{s:high-connectivity}

In this section, we study the connectivity of the tethered $X$-curve complex $TC_n(A,X)$ for an arbitrary decorated surface $A$ and many interesting decorated surfaces $X$, namely for:
\begin{itemizeb}
\item any $X$ of finite type in \S\ref{stab:sft};
\item any telescope $X=\mathfrak{L}(\Sigma)$ in \S\ref{stab:sls};
\item any binary tree surface $X=\mathfrak{B}(\Sigma)$ in \S\ref{stab:sbs}.
\end{itemizeb}
The results in \S\ref{stab:sft} are essentially classical, the only difference being that we allow the starting surface $A$ to be of infinite type, which does not essentially change the arguments. An overview of the strategy of the proofs in \S\ref{stab:sls} and \S\ref{stab:sbs} is given in Figure \ref{fig:structure-of-proof}.

\begin{figure}[tb]
    \centering
    \begin{tikzpicture}
        [x=1mm,y=1mm]
        \node (W) at (30,20) {$W_n(A,X)_\bullet$};
        \node (S) at (30,0) {$S_n(A,X)$};
        \node (U) at (30,-20) {$U_n(A,X)$};
        \node (TC) [anchor=west] at (U.east) {$\hspace{-1.5ex} {} \hyperref[lem-TC-U-cplx-iso]{\cong} TC_n(A,X)$};
        \draw[<->,color=blue] (W) to[out=240,in=120] node[right,font=\footnotesize,align=center]{\hyperref[prop:from-W-to-S]{homeomorphic} \\ \hyperref[prop:from-W-to-S]{realisations}} (S);
        \draw[->>] (S) to node[left,font=\footnotesize,color=blue,align=center]{\hyperref[obs-cjoin]{complete} \\ \hyperref[obs-cjoin]{join}} (U);

        \begin{scope}[yshift=5mm]
        \node (TCL) at (85,-10) {$TC_n(A,\mathfrak{L}(\Sigma))$};
        \node (LS) [color=blue,font=\footnotesize] at (63,-15) {\hyperref[stab:sls]{$X = \mathfrak{L}(\Sigma)$}};
        \node (TLC) at (85,5) {$TLC_n(A,\Sigma)$};
        \node (TLCbar) at (120,5) {$\overline{TLC}_n(A,\Sigma)$};
        \node (LC) at (120,-10) {$LC_n(A,\Sigma)$};
        \node at (85,-2.5) [font=\footnotesize,anchor=west] {$[n-2]$};
        \node (conn) [font=\footnotesize,color=blue] at (110,-20) {\hyperref[prop-conn-tlc]{$(n-2)$-connected}};
        \inclusion[left]{left}{\footnotesize\textcolor{blue}{\hyperref[thm-conn-lins]{$(*)$}}}{(TLC)}{(TCL)}
        \inclusion[right]{above}{\footnotesize\textcolor{blue}{\hyperref[prop-conn-tlc]{$(*)$}}}{(TLC)}{(TLCbar)}
        \draw[->>] (TLCbar) to node[right,font=\footnotesize,color=blue,align=center]{\hyperref[prop-conn-tlc]{homotopy} \\ \hyperref[prop-conn-tlc]{equiv.}} (LC);
        \draw[->,color=blue] (conn) to (LC);
        \end{scope}

        \begin{scope}[yshift=-5mm]
        \node (TCB) at (85,-30) {$TC_n(A,\mathfrak{B}(\Sigma))$};
        \node (BS) [color=blue,font=\footnotesize] at (63,-25) {\hyperref[stab:sbs]{$X = \mathfrak{B}(\Sigma)$}};
        \node (TBC) at (85,-45) {$bTLC_n(A,\Sigma)$};
        \node (TBCinfty) at (120,-45) {$bTLC_n^\infty(A,\Sigma)$};
        \node (TCBinfty) at (120,-30) {$TC_n^\infty(A,\mathfrak{B}(\Sigma))$};
        \node (cont) [font=\footnotesize,color=blue] at (128,-56) {\hyperref[prop-TLCn-inf-cont]{contractible}};
        \incl{(TBC)}{(TCB)}
        \inclusion{right}{\footnotesize\textcolor{blue}{\hyperref[prop:TC-inf-contr]{$(*)$}}}{(TBCinfty)}{(TCBinfty)}
        \draw[<->,color=blue] (TBC) to[out=-30,in=210] node[below,font=\footnotesize]{\hyperref[rem:same-finite-skeleton-TLC]{same $(n-2)$-skeleton}} (TBCinfty);
        \draw[<->,color=blue] (TCB) to[out=30,in=150] node[above,font=\footnotesize]{\hyperref[lem:TC-inf-share-skl]{same $(n-2)$-skeleton}} (TCBinfty);
        \draw[->,color=blue] (cont) to (TBCinfty);
        \end{scope}

        \draw[<->,color=blue] (TC.north east) to (TCL);
        \draw[<->,color=blue] (TC.south east) to (TCB);
    \end{tikzpicture}
    \caption{The strategy of the high-connectivity proofs for $X=\mathfrak{L}(\Sigma)$ and for $X=\mathfrak{B}(\Sigma)$. The symbol \textcolor{blue}{$(*)$} on an inclusion indicates a \emph{bad simplex argument}. The label $[n-2]$ on one of the inclusions indicates that only the $(n-2)$-skeleton includes into the larger complex.}
    \label{fig:structure-of-proof}
\end{figure}
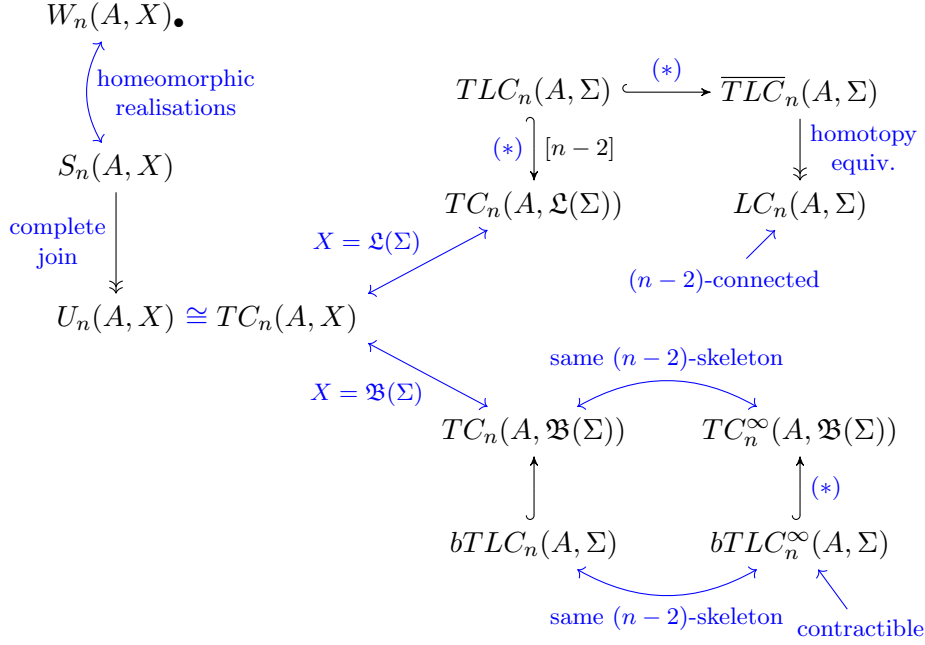

\subsection{Stabilisation by finite-type surfaces}
\label{stab:sft}

The results in this subsection are well-known when $A$ is of finite type. We indicate how this can be generalised to the infinite-type setting.

\begin{thm}
\label{thm-conn-fin}
Let $A$ be a decorated surface. Then:
\begin{enumerate}
\item\label{thm-conn-fin-pdisk} If $X$ is the once-punctured disc, then $TC_n(A,X)$ is $(n-2)$-connected. 
\item\label{thm-conn-fin-torus} If $X$ is the one-holed torus, then $TC_n(A,X)$ is $(\frac{g_A+n-3}{2})$-connected, where $g_A$ is the genus of $A$.
\end{enumerate}
\end{thm}

\begin{proof}
We consider each part in turn.
\begin{enumerate}
\item  Consider the surface $A\natural X^{\natural n}$, which has $n$ special punctures  coming from $X^{\natural n}$. According to Definition \ref{defn-teth-X-cur}, the vertices of $TC_n(A,X)$ are isotopy classes of tethered $X$-curves. Since $X$ is a punctured disc, just as in the proof of \cite[Lemma 5.21]{RWW17}, we may identify the complex $TC_n(A,X)$ with the arc complex $\mathcal{A}_n(A,X)$ with:
    \begin{itemize}
    \item vertices: isotopy classes of embedded arcs connecting the basepoint to an isolated end, which could lie in $X^{\natural n}$ or in $A$;
    \item $p$-simplices: a $(p+1)$-tuple of isotopy classes of arcs spans a $p$-simplex if there exist representative arcs that are pairwise disjoint outside the basepoint.
    \end{itemize}
When $A$ has $k$ isolated ends with $k< \infty$, this simplicial complex is shown to be Cohen-Macaulay of dimension $k+n-1$, in particular $(k+n-2)$-connected, in \cite[Proposition~7.2]{HW10}. (See also \cite[Theorem~2.3]{Tr14} and \cite[Theorem~3.2]{Tietz} for generalisations. When $A$ is the $2$-disc, this complex is actually contractible, by \cite[Theorem~2.48]{Damiolini}.) When $A$ has infinitely many isolated ends, we may choose an increasing family $\{A_i\}_{i\in \bN}$ of finite-type subsurfaces whose union is $A$, each $A_i$ contains the based boundary of $A$ and the inclusion maps induce injections on fundamental groups. Thus each $TC_n(A_i,X)$ is at least $(n-2)$-connected by the previous case. Consider a map $f \colon S^k \to TC_n(A,X)$ for $k\leq n-2$. Since $S^k$ is compact, its image must lie in the subcomplex $TC_n(A_i,X)$ for some $i$. So $f$ is nullhomotopic in $TC_n(A_i,X)$. Thus we have $[f]=0 \in \pi_k(TC_n(A,X))$. Hence we have shown that $TC_n(A,X)$ is $(n-2)$-connected when $A$ has infinitely many isolated ends. This proves part \eqref{thm-conn-fin-pdisk} of the theorem.

\item For part \eqref{thm-conn-fin-torus}, it is more convenient to use the loop model, i.e.~the complex $U_n(A,X)$ (see Definition \ref{defn:complex-Un(A,X)}). To show that $U_n(A,X)$ is highly-connected, we will consider the tethered chain complex $TCh(A\natural X^{\natural n},P)$ from \cite[\S 5.3]{HatcherVogtmann2017}. Here $P$ is an open interval in the based boundary containing the basepoint. To simplify, we may in fact assume that $P$ is the basepoint and denote the complex by $TCh(A\natural X^{\natural n})$. Recall that a \emph{chain} is a pair $(a,b)$ of curves intersecting transversely at a single point, together with an orientation of $b$; a \emph{tethered chain} is the union of a chain and a tether, where the tether is an arc connecting the basepoint with the positive side of the (oriented) $b$-curve of the chain; the interior of the tether is required to be disjoint from the chain. The complex $TCh(A\natural X^{\natural n})$ is then defined as follows:
    \begin{itemizeb}
    \item vertices: isotopy classes of tethered chains in $A \natural X^{\natural n}$;
    \item $p$-simplices: a $(p+1)$-tuple of isotopy classes of tethered chains spans a $p$-simplex if there exist representatives that are pairwise disjoint outside the basepoint.
    \end{itemizeb}
We claim that $TCh(A\natural X^{\natural n})$ is $(\frac{g_A+n-3}{2})$-connected. This holds if $A$ is of finite type by \cite[Proposition 5.5]{HatcherVogtmann2017}. If $A$ is of infinite type, we may do a similar trick as in part \eqref{thm-conn-fin-pdisk}. Let us choose an increasing family $\{A_i\}_{i\in \bN}$ of finite-type subsurfaces whose union is $A$, each $A_i$ contains the based boundary of $A$, the inclusion maps induce injections on fundamental groups and:
    \begin{itemizeb}
    \item the genus $g_{A_1}$ of $A_1$ equals $g_A$ if $A$ has finite genus;
    \item $\lim_{i\to \infty} (g_{A_i}) = \infty$ if $g_A=\infty$.
    \end{itemizeb}
Each $TCh(A_i\natural X^{\natural n})$ is $\bigl(\frac{g_{A_i} + n-3}{2}\bigr)$-connected by \cite[Proposition 5.5]{HatcherVogtmann2017}. Consider a map $f \colon S^k \to TCh(A\natural X^{\natural n})$ for $k\leq \frac{g_A+n-3}{2}$. Since $S^k$ is compact, its image must lie in the subcomplex $TCh(A_i\natural X^{\natural n})$ for some $i$. By increasing $i$ if necessary, we may ensure that $k\leq \frac{g_{A_i}+n-3}{2}$ so that $f$ is nullhomotopic in $TCh(A_i\natural X^{\natural n})$. It thus follows that we also have $[f]=0 \in \pi_k(TCh(A\natural X^{\natural n}))$. Hence $TCh(A\natural X^{\natural n})$ is $(\frac{g_A+n-3}{2})$-connected.

Given any tethered chain, one may take a small tubular neighbourhood of it; the boundary of this tubular neighbourhood is a based loop cutting off a torus from $A\natural X^{\natural n}$, so its isotopy class is a vertex in $U_n(A,X)$. One may check that this construction describes a simplicial map $TCh(A \natural X^{\natural n}) \to U_n(A,X)$ that is a complete join (a simplex of $U_n(A,X)$ is represented by a collection of based loops that cut off disjoint tori; lifting this to a simplex of $TCh(A\natural X^{\natural n})$ consists in choosing a ``core'' tethered chain for each torus, which may be chosen independently for each vertex of the simplex). By Remark \ref{rem-cjoin}, the complex $U_n(A,X)$ is a retract of $TCh(A\natural X^{\natural n})$ and hence it is also $(\frac{g_A+n-3}{2})$-connected. By Lemma \ref{lem-TC-U-cplx-iso} it then follows that $TC_n(A,X)$ is $(\frac{g_A+n-3}{2})$-connected.\qedhere
\end{enumerate}
\end{proof}

\subsection{Stabilisation by telescopes.}
\label{stab:sls}

In this section, we prove that the tethered $X$-curve complex $TC_n(A,X)$ is highly-connected for $1$-holed linear surfaces $X = \mathfrak{L}(\Sigma)$ that are \emph{telescopes}, i.e.~that have the property that the special end of $\mathfrak{L}(\Sigma)$ (corresponding to the unique end of the linear tree $\mathfrak{L}$) is topologically distinguished in the end space of $\mathfrak{L}(\Sigma)$; see Definition \ref{defn:telescope}.

\begin{defn}
Let $\Sigma$ be a connected surface with empty boundary, possibly of infinite type. Recall from Definition \ref{def:linear-binary-tree-surfaces} the \emph{linear surface} $\mathrm{Gr}_\mathfrak{L}(\Sigma)$; we will consider its $1$-holed version $\mathfrak{L}(\Sigma) = \mathrm{Gr}_\mathfrak{L}(\Sigma)_\circ$ (see Notation \ref{notation:one-holed-tree-surfaces}). We note that the surface $\mathfrak{L}(\Sigma)$ may be realised as the union of a family of subsurfaces $\{L_n\}_{n\geq 0}$ constructed as follows:
\begin{itemizeb}
\item $L_0$ is a cylinder with a basepoint in one boundary component. We refer to this boundary as its \emph{based boundary};
\item $L_{n}$ is obtained from $L_{n-1}$ by attaching a $2$-holed copy of $\Sigma$ along the non-based boundary of $L_{n-1}$; the basepoint and based boundary of $L_n$ are inherited from $L_{n-1}$.
\end{itemizeb}
We refer to the non-based boundary components of the $L_n$ as the \emph{level curves} of $\mathfrak{L}(\Sigma)$. The end of $\mathfrak{L}(\Sigma)$ accumulated by the level curves is called the \emph{level end} of $\mathfrak{L}(\Sigma)$. If $E$ denotes the space of ends of $\Sigma$, then the space of ends of $\mathfrak{L}(\Sigma)$ is the one-point compactification $(E\omega)^+$ of the disjoint union $E\omega$ of countably infinitely many copies of $E$; the level end corresponds to the point at infinity.
\end{defn}

\begin{ex}\hspace{0pt}
\begin{enumerate}
    \item If $\Sigma = T^2$ is the torus, then $\mathrm{Gr}_\mathfrak{L}(T^2)$ is the \emph{Loch Ness monster surface}, which is illustrated in Figure \ref{fig:Loch-Ness-monster-surface}. Removing the interior of the disc on the left-hand side of the figure, we obtain $\mathfrak{L}(T^2)$. The level curves are also illustrated in the figure.
    \item If $\Sigma = \bR^2$ is the plane (once-punctured sphere), then $\mathfrak{L}(\bR^2)$ is the ($1$-holed) \emph{flute surface}.
\end{enumerate}
\end{ex}

Now for any decorated surface $A$ and $n\geq 1$, the surface $A\natural \mathfrak{L}(\Sigma)^{\natural n}$ inherits a level structure from the $n$ copies of $\mathfrak{L}(\Sigma)$; see Figure \ref{fig:Ln} for an example. Recall that a \emph{tethered curve} in a based surface is the union of a tether and a simple closed curve, where a \emph{tether} is an embedded arc connecting the basepoint of the surface to the curve and whose interior is disjoint from the curve. A \emph{tethered level curve} in a based surface with a level structure is a tethered curve where the curve is one of the level curves of the surface: it therefore consists in a (discrete) choice of one of the level curves together with an arc attaching it to the basepoint.

\begin{figure}[tb]
    \centering
    \includegraphics[scale=0.7]{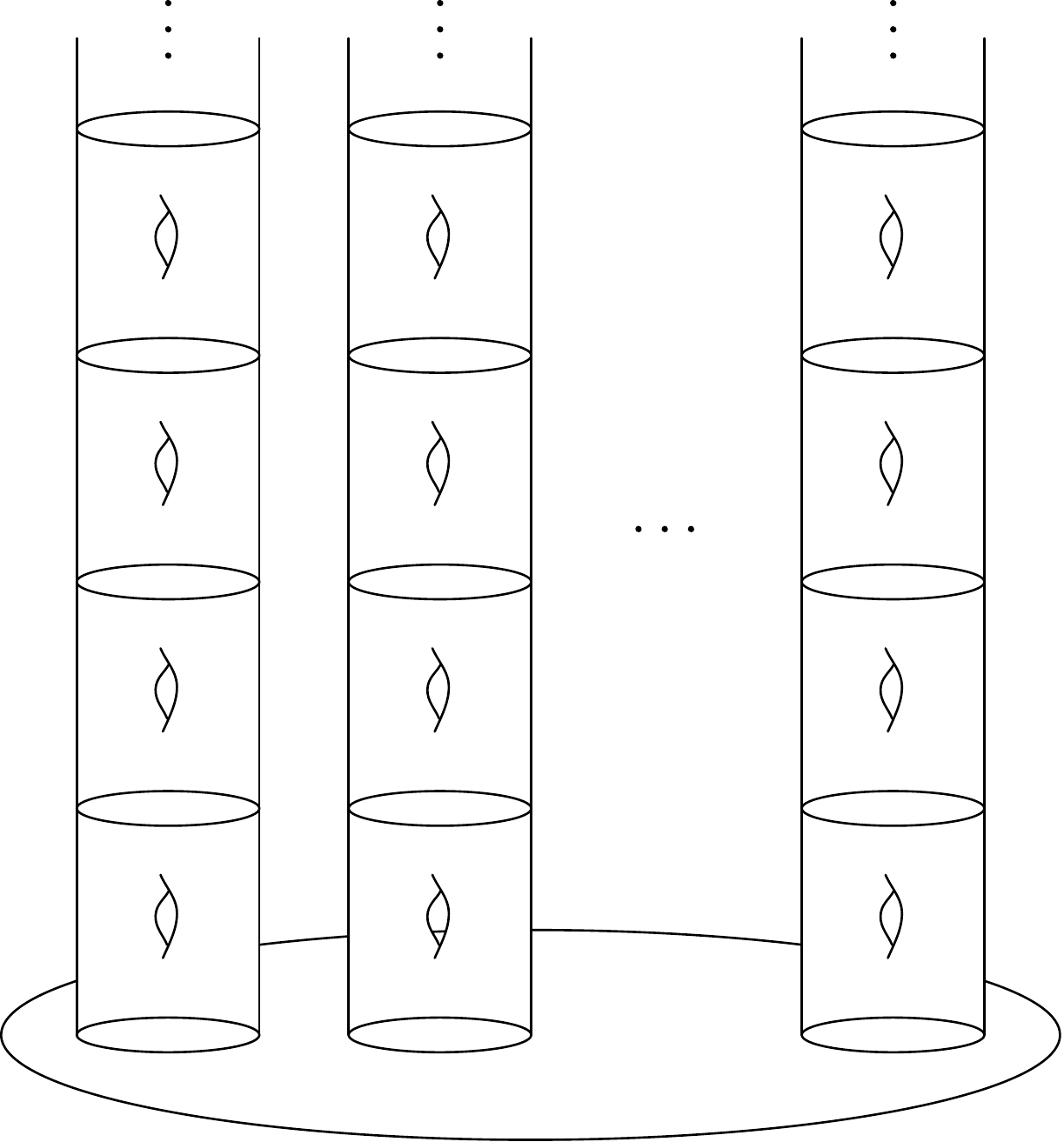}
    \caption{The level structure of $D^2\natural \mathfrak{L}(T^2)^{\natural n}$, inherited from that of $\mathfrak{L}(T^2)$.}
    \label{fig:Ln}
\end{figure}

\begin{defn}
\label{defn:TLC}
The \emph{tethered level curve complex} $TLC_n(A,\Sigma)$ is the simplicial complex with:
\begin{enumerate}
\item vertices: isotopy classes of tethered level curves in $A\natural \mathfrak{L}(\Sigma)^{\natural n}$;
\item $p$-simplices: a $(p+1)$-tuple of isotopy classes of tethered level curves spans a $p$-simplex if there exist representatives that are pairwise disjoint outside the basepoint.
\end{enumerate}
\end{defn}

More generally, we shall need to consider tethered level curve complexes defined on certain (``tame'') subsurfaces of $A\natural \mathfrak{L}(\Sigma)^{\natural n}$.

\begin{defn}
A subsurface $S \subseteq A\natural \mathfrak{L}(\Sigma)^{\natural n}$ is called \emph{tame} if it is connected, closed (as a subset), contains the basepoint of $A\natural \mathfrak{L}(\Sigma)^{\natural n}$ and each level end of $A\natural \mathfrak{L}(\Sigma)^{\natural n}$ has a neighbourhood that is either contained in $S$ or disjoint from $S$.

If $S$ is a tame subsurface of $A\natural \mathfrak{L}(\Sigma)^{\natural n}$, then a \emph{level end of $S$} means a level end of $A\natural \mathfrak{L}(\Sigma)^{\natural n}$ that is also an end of $S$. If a level curve of $A\natural \mathfrak{L}(\Sigma)^{\natural n}$ lies in $S$ and cuts off a level end of $S$ --- i.e.~if the corresponding level end of $A\natural \mathfrak{L}(\Sigma)^{\natural n}$ is also a (level) end of $S$ --- then we call it a \emph{level curve of $S$}. In this way, $S$ inherits a level structure from $A\natural \mathfrak{L}(\Sigma)^{\natural n}$.

We may therefore generalise Definition \ref{defn:TLC} to tame subsurfaces $S$ of $A\natural \mathfrak{L}(\Sigma)^{\natural n}$. A \emph{tethered level curve of $S$} is a tethered curve in $S$ where the curve is one of the level curves of $S$. If $S$ has $k \in \{0,\ldots,n\}$ level ends, we denote by $TLC_k(S)$ the simplicial complex whose vertices are isotopy classes of tethered level curves in $S$ and where a tuple of vertices spans a simplex if there exist representatives that are pairwise disjoint outside of the basepoint.
\end{defn}

\begin{prop}
\label{prop-conn-tlc}
The simplicial complex $TLC_n(A,\Sigma)$ is $(n-2)$-connected.
\end{prop}

\begin{rem}
Clearly two vertices whose level curves lie in the same copy of $\mathfrak{L}(\Sigma)$ cannot be joined by an edge. The simplicial complex $TLC_n(A,\Sigma)$ therefore has dimension $n-1$, so Proposition \ref{prop-conn-tlc} implies that it is homotopy equivalent to a wedge of $(n-1)$-spheres.
\end{rem}

\begin{proof}[Proof of Proposition \ref{prop-conn-tlc}]
We will prove the following stronger statement: for any $n,k\geq 0$ and any tame subsurface $S$ of $A\natural \mathfrak{L}(\Sigma)^{\natural n}$ with $k$ level ends, the simplicial complex $TLC_k(S)$ is $(k-2)$-connected. To prove this statement, we will include $TLC_k(S)$ into a larger complex $\overline{TLC}_k(S)$, then project this complex onto the \emph{level curve complex} $LC_k(S)$. This is the complex with level curves of $S$ as vertices and $p+1$ of them span a $p$-simplex if the level ends that they cut off are pairwise distinct. We first observe that $LC_k(S)$ is $(k-2)$-connected, then show that the projection $\overline{TLC}_k(S) \to LC_k(S)$ is a homotopy equivalence using a fibering argument of Hatcher and Vogtmann~\cite[Corollary 2.7]{HatcherVogtmann2017} and then apply a bad simplices argument to show that $TLC_k(S)$ is also $(k-2)$-connected by induction on $k$.

\uline{The level curve complex:} The complex $LC_k(S)$ is the join of the $k$ (zero-dimensional) subcomplexes given by the level curves that cut off a fixed level end. Since connectivity plus $2$ is additive under joins, it follows that $LC_k(S)$ is $(k-2)$-connected.

\uline{The larger complex:} The complex $\overline{TLC}_k(S)$ has the same vertices as $TLC_k(S)$ but simplices are allowed to share level curves, in other words, $p+1$ isotopy classes of tethered level curves span a $p$-simplex if they admit representatives where each pair satisfies one of the following two conditions:
\begin{itemizeb}
\item they are pairwise disjoint except at the basepoint and (hence) their level curves cut off different level ends of $S$;
\item the interiors of their tethers are disjoint and their level curves are equal.
\end{itemizeb}

\uline{The projection:} The projection $\pi \colon \overline{TLC}_k(S) \to LC_k(S)$ forgets the tethers and maps a tethered level curve to its level curve. This is a simplicial map but it is not a complete join, since it fails to be injective on individual simplices in general. We instead want to apply \cite[Corollary 2.7]{HatcherVogtmann2017} here. The fibre $\pi^{-1}(y)$ of this projection over the barycentre $y$ of a simplex $\sigma$ of $LC_k(S)$ consists of all tether systems for $\sigma$, i.e.~tethers that are only attached to level curves in $\sigma$. Let us fix one system $\sigma_0 \in \pi^{-1}(y)$ that has one tether attached to each level curve in $\sigma$ (so it lies in the subcomplex $TLC_k(S)$, although we will not use this fact). We will deformation retract the fibre $\pi^{-1}(y)$ into the star of $\sigma_0$ using a Hatcher flow \cite{Hat91}, see also \cite[\S 2.4]{HatcherVogtmann2017}. Denote the tethers of $\sigma_0$ by $t_0,\ldots,t_p$. Given any other tether system lying in this fibre, we get rid of the intersections of its tethers with the $t_i$ by putting it into minimal position with respect to the $t_i$, then pushing the intersection points with $t_0$ towards the basepoint and discarding the closed curve part. In this way, step by step, we get rid of all intersection points of the tethers with $t_0$ and then $t_1, \ldots, t_p$. At the end, we have a tether system that is disjoint from each $t_i$, so it lies in the star of $\sigma_0$. Thus we have deformation retracted $\pi^{-1}(y)$ onto the star of $\sigma_0$, which is contractible, so $\pi^{-1}(y)$ is contractible for the barycentre $y$ of each simplex $\sigma$ of $LC_k(S)$. Hence the projection $\pi$ is a homotopy equivalence by \cite[Corollary 2.7]{HatcherVogtmann2017}. Thus $\overline{TLC}_k(S)$ is also $(k-2)$-connected.

\uline{The inclusion:} Finally, we analyse the connectivity of the inclusion map $TLC_k(S) \hookrightarrow \overline{TLC}_k(S)$ using a bad simplices argument. We call a simplex in $\overline{TLC}_k(S)$ \emph{bad} if each of its level curves has at least two tethers attached to it. Notice that this definition clearly satisfies conditions \eqref{bad-sim-1} and \eqref{bad-sim-2} of \S\ref{sub-bad-sim}; also notice that vertices cannot be bad. Given a bad $p$-simplex $\sigma$, the surface obtained by cutting $S$ open along a system of tethered level curves representing $\sigma$ has several components. Some components (more precisely, one component for each level curve that we have cut along) do not contain the basepoint, so they will not contribute to the good link $G_\sigma$. Some components do contain the basepoint but do not contain any level ends, so they will also not contribute to $G_\sigma$. In any case, if we denote by $S_1,\ldots,S_d$ the components that contain the basepoint, then each $S_i$ is a tame subsurface of $A\natural \mathfrak{L}(\Sigma)^{\natural n}$ and $G_\sigma$ may be identified with the join complex $\ast_{i=1}^d TLC_{k_i}(S_i)$, where $k_i$ is the number of level ends of $S_i$. As noted above, the components that are cut off by the level curves of $\sigma$ do not contain the basepoint and each removes one level end from $S$ when it is cut off. Each level curve of $\sigma$ corresponds to at least two vertices, since $\sigma$ is bad, so the number of level ends that are removed is at most $\frac{p+1}{2}$. Thus we have $\sum_{i=1}^d k_i \geq k-\frac{p+1}{2}$. We also have $k_i < k$ for each $i$, since at least one level end has been removed when passing from $S$ to $S_i$. By induction, we know that $TLC_{k_i}(S_i)$ is at least $(k_i - 2)$-connected. Since connectivity plus $2$ is additive under joins, we deduce that the connectivity of $G_\sigma$ is at least
\[
\sum_{i=1}^d k_i - 2 \geq k-\tfrac{p+1}{2}-2 \geq k-p-2,
\]
where the last inequality holds since $p\geq 1$ by our earlier observation that vertices cannot be bad. Therefore, by Theorem \ref{thm--bad-sim} and the fact that $\overline{TLC}_k(S)$ is $(k-2)$-connected, we conclude that $TLC_k(S)$ is also $(k-2)$-connected.
\end{proof}

We have not so far needed to assume that the level end of the linear surface $\mathfrak{L}(\Sigma)$ is topologically distinguished (see Definition \ref{defn:topologically-distinguished}) in the end space of $\mathfrak{L}(\Sigma)$. This assumption will now be used to deduce high-connectivity of the complex $TC_n(A,\mathfrak{L}(\Sigma))$.

\begin{rem}
A key observation about topologically distinguished ends is the following. If a surface $X$ has a topologically distinguished end, then $X\natural X$ cannot be homeomorphic to $X$. We also mention the following alternative characterisation of topologically distinguished points. A point $x \in E$ is topologically distinguished if and only if the following holds: if a neighbourhood $U$ of a different point $y\neq x$ in $E$ contains a homeomorphic copy of a neighbourhood of $x$, then $U$ contains $x$.
\end{rem}

\begin{cor}
\label{thm-conn-lins}
Let $\mathfrak{L}(\Sigma)$ be a telescope, i.e.~a linear surface whose level end is topologically distinguished. Assume that $\Sigma \neq \bS^2$. Then the complex $TC_n(A,\mathfrak{L}(\Sigma))$ is $(n-3)$-connected.
\end{cor}

\begin{proof}
Assume that $n\geq 2$ (otherwise the claim is trivial). A vertex of $TLC_n(A,\Sigma)$ is a tethered level curve in $A\natural \mathfrak{L}(\Sigma)^{\natural n}$. This cuts off from $A\natural \mathfrak{L}(\Sigma)^{\natural n}$ a subsurface homeomorphic to $\mathfrak{L}(\Sigma)$ and the complementary surface is homeomorphic to $A\natural \mathfrak{L}(\Sigma)^{\natural (n-1)}$; to prove this latter fact we use the fact that $n\geq 2$ to ensure that we may ``absorb'' a finite connected sum $\Sigma \sharp \cdots \sharp \Sigma$ into one of the remaining copies of $\mathfrak{L}(\Sigma)$. The isotopy class of this tethered curve is therefore a vertex of $TC_n(A,\mathfrak{L}(\Sigma))$. This gives us a map of vertex sets
\begin{equation}
\label{eq:map-of-vertex-sets}
TLC_n(A,\Sigma)_0 \longrightarrow TC_n(A,\mathfrak{L}(\Sigma))_0,
\end{equation}
which is injective since distinct level curves are not isotopic (because we have assumed that $\Sigma \neq \bS^2$). For any $p$-simplex of $TLC_n(A,\Sigma)$ with $p\leq n-2$, the images of its vertices in $TC_n(A,\mathfrak{L}(\Sigma))_0$ span a $p$-simplex by the same argument as above for vertices, using the fact that $n-p-1\geq 1$ in order to be able to ``absorb'' a finite connected sum $\Sigma \sharp \cdots \sharp \Sigma$ into one of the remaining copies of $\mathfrak{L}(\Sigma)$. Thus we have an embedding of simplicial complexes
\begin{equation}
\label{eq:embedding-of-complexes}
TLC_n(A,\Sigma)^{(n-2)} \lhook\joinrel\longrightarrow TC_n(A,\mathfrak{L}(\Sigma)),
\end{equation}
where on the left we have the $(n-2)$-skeleton of $TLC_n(A,\Sigma)$. Note that the $(n-1)$-simplices of $TLC_n(A,\Sigma)$ are not all sent to simplices of $TC_n(A,\mathfrak{L}(\Sigma))$ under \eqref{eq:map-of-vertex-sets}, so the embedding \eqref{eq:embedding-of-complexes} does not extend to the whole complex $TLC_n(A,\Sigma)$. Nevertheless, the fact that $TLC_n(A,\Sigma)$ is $(n-2)$-connected (Proposition \ref{prop-conn-tlc}) implies that its $(n-2)$-skeleton is $(n-3)$-connected. We will now use a bad simplices argument for the inclusion
\begin{equation}
\label{eq:embedding-of-complexes-restricted}
TLC_n(A,\Sigma)^{(n-2)} \lhook\joinrel\longrightarrow TC_n(A,\mathfrak{L}(\Sigma))^{(n-2)},
\end{equation}
where we have restricted the right-hand side of \eqref{eq:embedding-of-complexes} to its $(n-2)$-skeleton, to deduce that $TC_n(A,\mathfrak{L}(\Sigma))$ is also $(n-3)$-connected.

We say a simplex in $TC_n(A,\mathfrak{L}(\Sigma))^{(n-2)}$ is \emph{bad} if none of its vertices lie in $TLC_n(A,\Sigma)^{(n-2)}$. This means, equivalently, that none of its vertices, which are (isotopy classes of) tethered curves, is isotopic to a tethered \emph{level} curve. Let $\sigma$ be a bad $p$-simplex and $G_\sigma$ its good link. We need to show that $G_\sigma$ is at least $(n-p-4)$-connected. Each of the $p+1$ tethered curves of $\sigma$ cuts off from $A\natural \mathfrak{L}(\Sigma)^{\natural n}$ one homeomorphic copy of $\mathfrak{L}(\Sigma)$. Since the level end of $\mathfrak{L}(\Sigma)$ is topologically distinguished, each homeomorphic copy of $\mathfrak{L}(\Sigma)$ that is cut off contains at most one level end of $A\natural \mathfrak{L}(\Sigma)^{\natural n}$ (it may contain zero, if $A$ also contains homeomorphic copies of $\mathfrak{L}(\Sigma)$). Denote by $S_\sigma$ the complementary surface obtained after cutting off all subsurfaces bounded by the tethered curves of $\sigma$ and denote by $n'$ the number of level ends of $A\natural \mathfrak{L}(\Sigma)^{\natural n}$ contained in $S_\sigma$. By the preceding argument we know that $n' \geq n-p-1$. There is a decomposition $S_\sigma \cong A' \natural \mathfrak{L}(\Sigma)^{\natural n'}$ of based surfaces equipped with level structures, where the right-hand side has the standard level structure inherited from $\mathfrak{L}(\Sigma)$ and the level structure of $S_\sigma$ consists of all level curves of $A\natural \mathfrak{L}(\Sigma)^{\natural n}$ that are contained in $S_\sigma$. Note that $A'$ may be different from $A$ (part of $A$ may have been cut off and we may have absorbed some copies of $\Sigma$ into $A'$). Under this identification, the good link $G_\sigma$ identifies with the complex $TLC_{n'}(A',\Sigma)^{(n'-2)}$. This has connectivity at least $n'-3 \geq n-p-4$ as it is the $(n'-2)$-skeleton of $TLC_{n'}(A',\Sigma)$, which is $(n'-2)$-connected  by Proposition \ref{prop-conn-tlc}. Theorem \ref{thm--bad-sim} therefore implies that $TC_n(A,\mathfrak{L}(\Sigma))^{(n-2)}$ is $(n-3)$-connected, thus the whole complex $TC_n(A,\mathfrak{L}(\Sigma))$ is also $(n-3)$-connected.
\end{proof}

\begin{cor}
\label{thm-conn-lins-shifted}
Let $\mathfrak{L}(\Sigma)$ be a telescope, i.e.~a linear surface whose level end is topologically distinguished. Assume that $\Sigma \neq \bS^2$. Then the complex $TC_n(A\natural \mathfrak{L}(\Sigma),\mathfrak{L}(\Sigma))$ is $(n-2)$-connected.
\end{cor}
\begin{proof}
This follows immediately from Corollary \ref{thm-conn-lins} and the observation that $TC_n(A\natural \mathfrak{L}(\Sigma),\mathfrak{L}(\Sigma))$ is isomorphic to the $(n-1)$-skeleton of $TC_{n+1}(A,\mathfrak{L}(\Sigma))$.
\end{proof}

\subsection{Stabilisation by binary tree surfaces.}
\label{stab:sbs}

In this subsection, we prove that the tethered $X$-curve complex $TC_n(A,X)$ is highly-connected for $1$-holed binary tree surfaces $X = \mathfrak{B}(\Sigma)$.

\begin{defn}
Let $\Sigma$ be a connected surface with empty boundary, possibly of infinite type. Recall from Definition \ref{def:linear-binary-tree-surfaces} the \emph{binary tree surface} $\mathrm{Gr}_\mathfrak{B}(\Sigma)$; we will consider its $1$-holed version $\mathfrak{B}(\Sigma) = \mathrm{Gr}_\mathfrak{B}(\Sigma)_\circ$ (see Notation \ref{notation:one-holed-tree-surfaces}). We note that the surface $\mathfrak{B}(\Sigma)$ may be realised as the union of a family of subsurfaces $\{C_n\}_{n\geq 0}$ constructed as follows:
\begin{itemizeb}
\item $C_0$ is a cylinder with a basepoint in one boundary component. We refer to this boundary as its \emph{based boundary};
\item $C_n$ is obtained from $C_{n-1}$ by attaching a $3$-holed copy of $\Sigma$ along each non-based boundary component of $C_{n-1}$. The basepoint and based boundary of $C_n$ are inherited from $C_{n-1}$. Note that $C_n$ has $2^n$ non-based boundary components.
\end{itemizeb}
We refer to the non-based boundary components of the $C_n$ as the \emph{level curves} of $\mathfrak{B}(\Sigma)$. The ends of $\mathfrak{B}(\Sigma)$ accumulated by the level curves form a Cantor set in the end space of $\mathfrak{B}(\Sigma)$. We refer to these ends as the \emph{level ends} of $\mathfrak{B}(\Sigma)$. If $E$ denotes the space of ends of $\Sigma$, then the space of ends of $\mathfrak{B}(\Sigma)$ is the Cantor compactification $(E\omega)^\cC$ (see Definition \ref{defn:Cantor-compactification}) of the disjoint union $E\omega$ of countably infinitely many copies of $E$; the level ends correspond to the Cantor set at infinity in $(E\omega)^\cC$.
\end{defn}

Now for any decorated surface $A$ and $n\geq 1$, the surface $A\natural \mathfrak{B}(\Sigma)^{\natural n}$ inherits a level structure from the $n$ copies of $\mathfrak{B}(\Sigma)$; see Figure \ref{fig:Bn} for an example. We define a binary-tree analogue of the tethered level curve complex of \S\ref{stab:sls} (Definition \ref{defn:TLC}).

\begin{figure}[tb]
    \centering
    \includegraphics[scale=0.7]{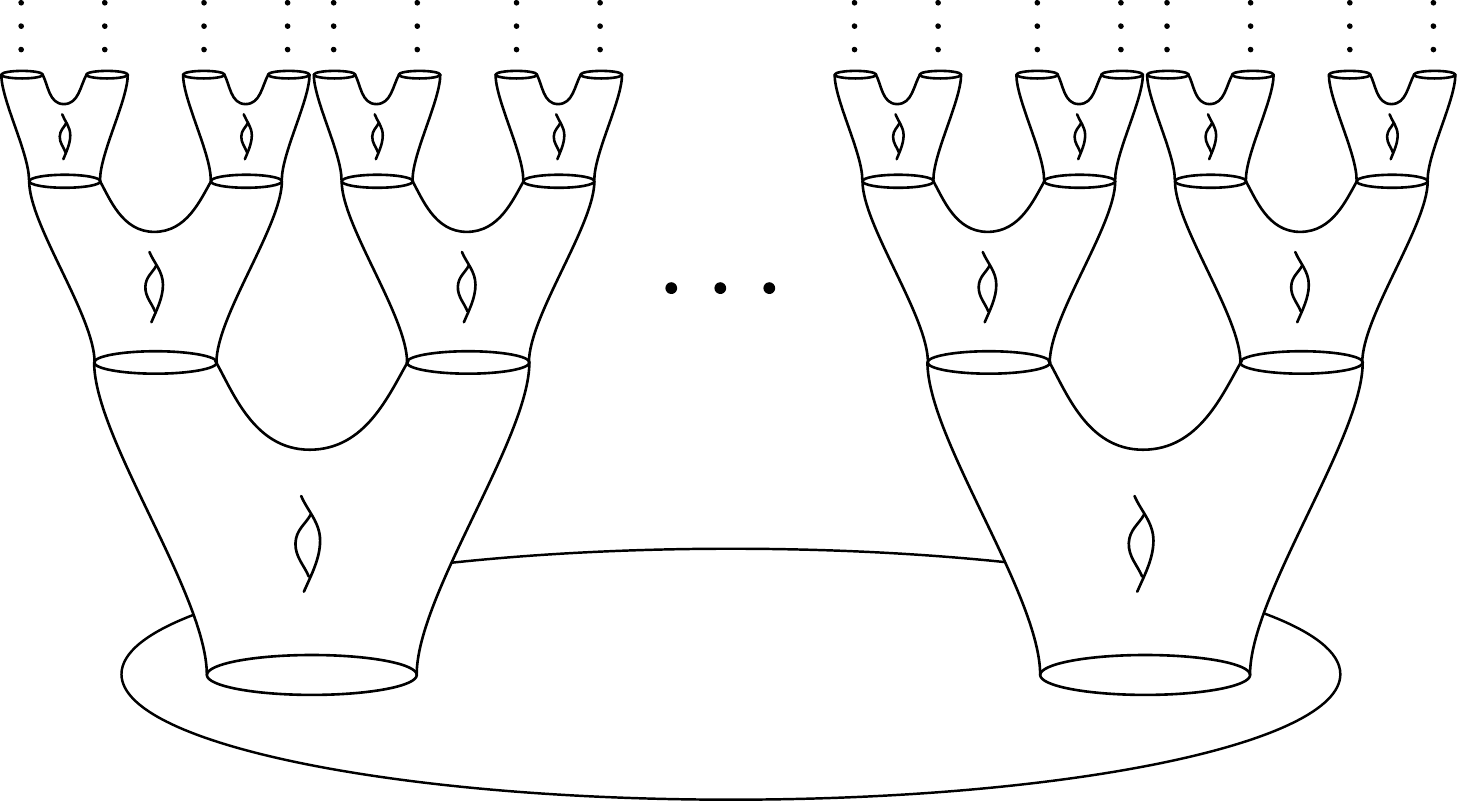}
    \caption{The level structure of $D^2\natural \mathfrak{B}(T^2)^{\natural n}$, inherited from that of $\mathfrak{B}(T^2)$.}
    \label{fig:Bn}
\end{figure}

\begin{defn}
The \emph{binary tethered level curve complex} $bTLC_n(A,\Sigma)$ is the simplicial complex with:
\begin{enumerate}
\item vertices: isotopy classes of tethered level curves in $A\natural \mathfrak{B}(\Sigma)^{\natural n}$;
\item $p$-simplices: a $(p+1)$-tuple of isotopy classes of tethered level curves spans a $p$-simplex if there exist representatives that are pairwise disjoint outside the basepoint and:
  \begin{itemizeb}
  \item (if $p<n-1$): some level ends of $A\natural \mathfrak{B}(\Sigma)^{\natural n}$ lie outside of the subsurfaces bounded by the $p+1$ tethered level curves;
  \item (if $p=n-1$): all level ends of $A\natural \mathfrak{B}(\Sigma)^{\natural n}$ lie inside the subsurfaces bounded by the $p+1$ tethered level curves;
  \end{itemizeb}
\item there are no $p$-simplices for $p\geq n$.
\end{enumerate}
\end{defn}

As in \cite[\S 2]{SW19} and  \cite[\S 4]{SkipperWu2021}, we define an intermediate complex which helps us to determine the connectivity of $bTLC_n(A,\Sigma)$.

\begin{defn}
The simplicial complex $bTLC_n^\infty(A,\Sigma)$ has:
\begin{enumerate}
\item vertices: isotopy classes of tethered level curves in $A\natural \mathfrak{B}(\Sigma)^{\natural n}$;
\item $p$-simplices: a $(p+1)$-tuple of isotopy classes of tethered level curves spans a $p$-simplex if there exist representatives that are pairwise disjoint outside the basepoint and some level ends of $A\natural \mathfrak{B}(\Sigma)^{\natural n}$ lie outside of the subsurfaces bounded by the $p+1$ tethered level curves.
\end{enumerate}
\end{defn}

\begin{rem}
\label{rem:same-finite-skeleton-TLC}
The definitions of $bTLC_n^\infty(A,\Sigma)$ and $bTLC_n(A,\Sigma)$ are identical for $p$-simplices with $p\leq n-2$, so they have the same $(n-2)$-skeleton. On the other hand $bTLC_n(A,\Sigma)$ is $(n-1)$-dimensional whereas $bTLC_n^\infty(A,\Sigma)$ is infinite-dimensional. The complex $bTLC_n^\infty(A,\Sigma)$ is therefore obtained from $bTLC_n(A,\Sigma)$ by throwing away its top-dimensional simplices and adding new $p$-simplices for all $p\geq n-1$.
\end{rem}

\begin{prop}
\label{prop-TLCn-inf-cont}
Let $\Sigma$ be any connected surface without boundary and let $A$ be any decorated surface. Then the simplicial complex $bTLC_n^\infty(A,\Sigma)$ is contractible for any $n\geq 1$.
\end{prop}

\begin{proof}
The proof is a verbatim adaptation of the proof of \cite[Proposition~4.26]{SkipperWu2021}, which in turn follows a similar strategy to the proof of \cite[Proposition~3.1]{SW19}. In fact, the special case $(A = \bD^2 , \Sigma = \bS^2)$ coincides precisely with the $d=2$ case of \cite[Proposition~4.26]{SkipperWu2021} (allowing $d\geq 2$ would correspond to replacing the infinite binary tree $\mathfrak{B}$ with the infinite $d$-ary tree). The proof carries over identically for the general case of $(A,\Sigma)$.
\end{proof}

It follows by Remark \ref{rem:same-finite-skeleton-TLC} that $bTLC_n(A,\Sigma)$ is $(n-3)$-connected. However, instead of using this fact, we will go the other way (anticlockwise) around the square in the bottom-right of Figure \ref{fig:structure-of-proof} outlining the strategy of proof. To do this, we will pass via the following complex, which builds a bridge between $bTLC_n^\infty(A,\Sigma)$ and $TC_n(A,\mathfrak{B}(\Sigma))$.

\begin{defn}
We define the simplicial complex $TC_n^\infty(A,\mathfrak{B}(\Sigma))$ by:
\begin{enumerate}
\item vertices: isotopy classes of tethered curves $\alpha \colon L \hookrightarrow A\natural \mathfrak{B}(\Sigma)^{\natural n}$ such that $\alpha([1,2])$ bounds a subsurface homeomorphic to $\mathfrak{B}(\Sigma)$ and some of the level ends of $A\natural \mathfrak{B}(\Sigma)^{\natural n}$ lie outside of this subsurface.
\item $p$-simplices: a $(p+1)$-tuple of vertices spans a $p$-simplex if there exist representatives that are pairwise disjoint except at the basepoint and some of the level ends of $A\natural \mathfrak{B}(\Sigma)^{\natural n}$ lie outside of the subsurfaces bounded by the $p+1$ curves.
\end{enumerate}
\end{defn}

\begin{lem}
\label{lem:TC-inf-share-skl}
The $(n-2)$-skeletons of $TC_n(A,\mathfrak{B}(\Sigma))$ and of $TC_n^\infty(A,\mathfrak{B}(\Sigma))$ are isomorphic.
\end{lem}
\begin{proof}
Suppose that $n\geq 2$ (otherwise the statement is vacuous). Vertices of $TC_n(A,\mathfrak{B}(\Sigma))$ are tethered $\mathfrak{B}(\Sigma)$-curves, meaning tethered curves that cut off a subsurface homeomorphic to $\mathfrak{B}(\Sigma)$ and where the complementary surface is homeomorphic to $A\natural \mathfrak{B}(\Sigma)^{\natural (n-1)}$. Since $\mathfrak{B}(\Sigma)^{\natural m} \cong \mathfrak{B}(\Sigma)$ for all $m\geq 1$, the last part of this condition is equivalent to saying that the complementary surface is homeomorphic to $A\natural \mathfrak{B}(\Sigma)$. This, in turn, is equivalent to saying that the space of level ends of $A \natural \mathfrak{B}(\Sigma)^{\natural n}$ that are not bounded by the tethered curve is homeomorphic to the Cantor set. Now the space of level ends of $A \natural \mathfrak{B}(\Sigma)^{\natural n}$ that are not bounded by the tethered curve is always a clopen subspace of the space of all level ends, which is a Cantor set. All clopen subspaces of the Cantor set are either empty or homeomorphic to the Cantor set itself. Thus the condition is equivalent to saying that the space of level ends of $A \natural \mathfrak{B}(\Sigma)^{\natural n}$ that are not bounded by the tethered curve is non-empty. This is precisely the condition for a tethered curve that cuts off a subsurface homeomorphic to $\mathfrak{B}(\Sigma)$ to be a vertex of $TC_n^\infty(A,\mathfrak{B}(\Sigma))$. Hence $TC_n(A,\mathfrak{B}(\Sigma))$ and $TC_n^\infty(A,\mathfrak{B}(\Sigma))$ have the same set of vertices.

An identical argument shows that the $p$-simplices of $TC_n(A,\mathfrak{B}(\Sigma))$ and of $TC_n^\infty(A,\mathfrak{B}(\Sigma))$ are the same for $p\leq n-2$, since $\mathfrak{B}(\Sigma)^{\natural (n-p-1)} \cong \mathfrak{B}(\Sigma)$.
\end{proof}

\begin{prop}
\label{prop:TC-inf-contr}
Let $\Sigma$ be any connected surface without boundary and let $A$ be any decorated surface. Then the simplicial complex $TC_n^\infty(A, \mathfrak{B}(\Sigma))$ is contractible for any $n\geq 1$.
\end{prop}

\begin{proof}
Since the complex $bTLC_n^\infty(A,\Sigma)$ is isomorphic to the full subcomplex of $TC_n^\infty(A,\mathfrak{B}(\Sigma))$ spanned by all vertices (tethered curves) that are isotopic to tethered \emph{level} curves, we may use a bad simplices argument to deduce contractibility of $TC_n^\infty(A, \mathfrak{B}(\Sigma))$ from the contractibility of $bTLC_n^\infty(A,\Sigma)$ (Proposition \ref{prop-TLCn-inf-cont}).

Let us call a vertex of $TC_n^\infty(A,\mathfrak{B}(\Sigma))$ \emph{bad} if it does not lie in $bTLC_n^\infty(A,\Sigma)$, i.e.~its curve is not isotopic to a level curve, and we call a simplex \emph{bad} if all of its vertices are bad. Let $\sigma$ be a bad $p$-simplex and consider its good link $G_\sigma$. It will suffice to show that $G_\sigma$ is contractible, since Proposition \ref{prop-bad-sim} will then imply that the inclusion is a weak equivalence. Denote by $C_\sigma$ the complementary surface obtained from $A\natural \mathfrak{B}(\Sigma)^{\natural n}$ by cutting off all subsurfaces bounded by the tethered curves of $\sigma$. By definition of the simplices of $TC_n^\infty(A, \mathfrak{B}(\Sigma))$, some of the level ends of $A\natural \mathfrak{B}(\Sigma)^{\natural n}$ lie in $C_\sigma$. As in the previous proof, the collection of level ends lying in $C_\sigma$ is homeomorphic to a clopen subset of the Cantor set, hence homeomorphic to the Cantor set. Thus we may identify $C_\sigma \cong A' \natural \mathfrak{B}(\Sigma)$ as based surfaces equipped with level curves (the level curves in $C_\sigma$ being those inherited from $A\natural \mathfrak{B}(\Sigma)^{\natural n}$). The good link $G_\sigma$ then identifies with the complex $bTLC_1^\infty(A',\Sigma)$, which is contractible by Proposition \ref{prop-TLCn-inf-cont}.
\end{proof}

Combining Proposition \ref{prop:TC-inf-contr} and Lemma \ref{lem:TC-inf-share-skl}, we immediately deduce:

\begin{thm}
\label{thm-conn-btree}
Let $\Sigma$ be any connected surface without boundary and let $A$ be any decorated surface. Then the simplicial complex $TC_n(A,\mathfrak{B}(\Sigma))$ is $(n-3)$-connected.
\end{thm}

However, the statement that we will need is the following shifted version:

\begin{thm}
\label{thm-conn-btree-shifted}
Let $\Sigma$ be any connected surface without boundary and let $A$ be any decorated surface. Then the simplicial complex $TC_n(A\natural \mathfrak{B}(\Sigma),\mathfrak{B}(\Sigma))$ is $(n-2)$-connected.
\end{thm}
\begin{proof}
This will follow immediately from Proposition \ref{prop:TC-inf-contr} if we show that $TC_n(A\natural \mathfrak{B}(\Sigma),\mathfrak{B}(\Sigma))$ is isomorphic to the $(n-1)$-skeleton of $TC_{n+1}^\infty(A, \mathfrak{B}(\Sigma))$. This claim may be proven by the same argument as in the proof of Lemma \ref{lem:TC-inf-share-skl}, which now works up to one dimension higher since the condition for a $(p+1)$-tuple of vertices to span a $p$-simplex in $TC_n(A\natural \mathfrak{B}(\Sigma),\mathfrak{B}(\Sigma))$ is pairwise disjoint outside of the basepoint and that the complementary surface must be homeomorphic to $A\natural \mathfrak{B}(\Sigma) \natural \mathfrak{B}(\Sigma)^{\natural (n-p-1)}$, which is homeomorphic to $A\natural \mathfrak{B}(\Sigma)$ whenever $p\leq n-1$.
\end{proof}

\renewcommand{\theHsection}{appendixsection.\thesection}
\renewcommand{\thesection}{\Alph{section}}
\setcounter{section}{0}

\section{Appendix: homeomorphisms vs. diffeomorphisms}
\label{appendix-diff-homeo}

In this appendix, we include a sketch of proof of the following classical fact (used in \S\ref{s:hs} to pass freely between topological and smooth mapping class groups), to emphasise that it holds for non-compact as well as compact surfaces.

\begin{lem}[Lemma \ref{diff-to-homeo}]
For any smooth surface $S$, the forgetful map
\begin{equation}
\label{eq:forgetful-map}
\mathrm{Diff}(S,\partial S) \longrightarrow \mathrm{Homeo}(S,\partial S)
\end{equation}
is a weak homotopy equivalence of topological groups.
\end{lem}
\begin{proof}[Sketch of proof]
Let us first assume that $\partial S = \varnothing$. It will suffice to show that the quotient space $\mathrm{Homeo}(S)/\mathrm{Diff}(S)$, which is the homotopy fibre of the induced map $B\mathrm{Diff}(S) \to B\mathrm{Homeo}(S)$ of classifying spaces, is weakly contractible. By smoothing theory \cite[Essay~V]{KirbySiebenmann}, it is weakly homotopy equivalent to the union of some path-components of the space of sections of a fibre bundle over $S$ with fibre $\mathrm{Homeo}(\bR^2)/\mathrm{Diff}(\bR^2) \simeq \mathrm{Homeo}(\bR^2)/O(2)$. By obstruction theory, this section space will be weakly contractible as long as $\mathrm{Homeo}(\bR^2)/O(2)$ is weakly contractible, equivalently the inclusion $O(2) \hookrightarrow \mathrm{Homeo}(\bR^2)$ induces isomorphisms on $\pi_i$ for all $i\geq 0$. For $i=0$ it is clearly injective and surjectivity follows from the fact that every orientation-preserving homeomorphism of $\bR^2$ is \emph{stable} (this is the stable homeomorphism theorem in dimension $2$, which follows by \cite{BrownGluck1964} from the annulus theorem in dimension $2$, due to Rad{\'o}) and thus isotopic to the identity. For $i\geq 1$ this is proven in \cite[Theorem~1.1]{Yagasaki2000}.

We have thus proven that the inclusion $\mathrm{Diff}(S) \hookrightarrow \mathrm{Homeo}(S)$ is a weak homotopy equivalence for any surface $S$ without boundary. A similar (but easier) argument implies the corresponding statement for $1$-manifolds, using the basic fact that $\mathrm{Homeo}^+(\bR)$ is contractible. From these two facts, the restriction fibrations $\mathrm{Diff}(S) \to \mathrm{Diff}(\partial S)$ and $\mathrm{Homeo}(S) \to \mathrm{Homeo}(\partial S)$ and the five-lemma, we deduce that \eqref{eq:forgetful-map} is a weak homotopy equivalence also when $\partial S \neq \varnothing$.
\end{proof}

\begin{rem}
An alternative reference for smoothing theory, which allows manifolds to have non-empty boundary, is \cite{BurgheleaLashof}. However, it assumes that the underlying manifold is compact, so it does not apply to infinite-type surfaces.
\end{rem}

\begin{rem}
One could try to prove that $\mathrm{Homeo}(\bR^2)/O(2)$ is weakly contractible by applying Morlet's weak equivalence $\Omega^3(\mathrm{Homeo}(\bR^2)/O(2)) \simeq \mathrm{Diff}(\bD^2,\partial \bD^2)$ (which is smoothing theory applied to the $2$-disc) together with Smale's theorem \cite{Smale1959} that $\mathrm{Diff}(\bD^2,\partial \bD^2)$ is contractible. However, this only implies that $\pi_i(\mathrm{Homeo}(\bR^2)/O(2)) = 0$ for $i\geq 3$.
\end{rem}

\bibliographystyle{alpha}
\bibliography{references.bib}

\end{document}